\documentclass[11pt, letterpaper,reqno]{amsart}
\usepackage[left=1in,right=1in,bottom=0.5in,top=0.8in]{geometry}
\usepackage{amsfonts}
\usepackage{amsmath, amssymb}
\usepackage{graphicx}
\usepackage[font=small,labelfont=bf]{caption}
\usepackage{epstopdf}
\usepackage{xcolor}
\usepackage{amsthm}
\usepackage{float}
\usepackage{pgfplots}
\usepackage{listings}
\usepackage{longtable}
\usepackage{mathrsfs}
\usepackage{dsfont}
\usepackage[pdfpagelabels,hyperindex]{hyperref}
\hypersetup{linkbordercolor=green}

\newcommand*{\rom}[1]{\expandafter\@slowromancap\romannumeral #1@}
 \colorlet{lgray}{white!80!black}
\colorlet{lred}{white!85!red}
\colorlet{lgreen}{white!60!green}
\colorlet{dgreen}{black!30!green}
\colorlet{lpurple}{white!60!purple}
\colorlet{lblue}{white!60!blue}
\definecolor{green}{rgb}{0.1,0.8,0.1}
\definecolor{yellow}{rgb}{1.0,0.85,0.25}
\definecolor{purple}{rgb}{1.0, 0, 1.0}
\definecolor{blue}{rgb}{0, 0, 1.0}
\tikzstyle{unfused}=[lgray, line width=1.5pt, ->]
\tikzstyle{fused}=[lgray, line width=4pt, ->]
\tikzstyle{dual}=[black, line width=1pt, dashed]
\tikzstyle{lightdual}=[black, line width=0.5pt, dashed]
\tikzstyle{cut}=[black, line width=1.0pt]

\theoremstyle{plain}
\newtheorem{theorem}{Theorem}[section]

\newtheorem{proposition}[theorem]{Proposition}
\newtheorem{corollary}[theorem]{Corollary}
\theoremstyle{remark}
\newtheorem{definition}[theorem]{Definition}
\newtheorem{remark}[theorem]{Remark}

\newtheorem{fact}[theorem]{Fact}
\newenvironment{claim}[1]{\par\noindent\underline{Claim:}\space#1}{}

\newcommand{\ARXIV}[1]{\href{https://arXiv.org/abs/#1}{arXiv:#1}}

\colorlet{shadecolor}{gray!20}
\pgfplotsset{compat=1.9}
\usetikzlibrary{shapes.multipart}
\usetikzlibrary{patterns}
\usetikzlibrary{shapes.multipart}
\usetikzlibrary{arrows}
\usetikzlibrary{decorations.markings}
\usepgflibrary{decorations.shapes}
\usetikzlibrary{decorations.shapes}
\usepgflibrary{shapes.symbols}
\usetikzlibrary{shapes.symbols}
\usetikzlibrary{decorations.pathreplacing}
\tikzstyle{fleche}=[>=stealth', postaction={decorate}, thick]
\tikzstyle{axis}=[->, >=stealth', thick, gray]
\tikzstyle{paths}=[>->, >=stealth', thick]
\tikzstyle{path}=[->, >=stealth', thick]
\tikzstyle{grille}=[dotted, gray]
\newcommand{\paththick}{\raisebox{-6pt}{\begin{tikzpicture}[scale=0.35]
		\draw[thick] (-1,0) -- (1,0);
		\draw[thick] (0,-1) -- (0,1);
		\end{tikzpicture}}}
\newcommand{\pathdotted}{\raisebox{-6pt}{\begin{tikzpicture}[scale=0.35]
		\draw[dotted] (-1,0) -- (1,0);
		\draw[dotted] (0,-1) -- (0,1);
		\end{tikzpicture}}}
\newcommand{\pathrr}{\raisebox{-6pt}{\begin{tikzpicture}[scale=0.35]
		\draw[thick] (-1,0) -- (1,0);
		\draw[dotted] (0,-1) -- (0,1);
		\end{tikzpicture}}}
\newcommand{\pathru}{\raisebox{-6pt}{\begin{tikzpicture}[scale=0.35]
		\draw[thick] (-1,0) -- (0,0) -- (0,1);
		\draw[dotted] (0,-1) -- (0,0) -- (1,0);
		\end{tikzpicture}}}
\newcommand{\pathuu}{\raisebox{-6pt}{\begin{tikzpicture}[scale=0.35]
		\draw[dotted] (-1,0) -- (1,0);
		\draw[thick] (0,-1) -- (0,1);
		\end{tikzpicture}}}
\newcommand{\pathur}{\raisebox{-6pt}{\begin{tikzpicture}[scale=0.35]
		\draw[dotted] (-1,0) -- (0,0) -- (0,1);
		\draw[thick] (0,-1) -- (0,0) -- (1,0);
		\end{tikzpicture}}}
\newcommand{\pathbrr}{\raisebox{2pt}{\begin{tikzpicture}[scale=0.5]
		\draw[path] (-1,0) -- (0,0);
		\draw[dotted] (0,0) -- (0,1);
		\end{tikzpicture}}}
\newcommand{\pathbuu}{\raisebox{2pt}{\begin{tikzpicture}[scale=0.5]
		\draw[path] (0,0) -- (0,0.1);
		\draw[thick] (0,0) -- (0,1);
		\draw[dotted] (-1,0) -- (0,0);
		\end{tikzpicture}}}
\newcommand{\pathbru}{\raisebox{2pt}{\begin{tikzpicture}[scale=0.5]
		\draw[thick] (-1,0) -- (0,0);
		\draw[thick] (0,0) -- (0,1);
		\end{tikzpicture}}}
\newcommand{\pathbur}{\raisebox{2pt}{\begin{tikzpicture}[scale=0.5]
		\draw[dotted] (0,0) -- (0,1);
		\draw[dotted] (-1,0) -- (0,0);
		\end{tikzpicture}}} 
\newcommand{\pathlyn}{\raisebox{2pt}{\begin{tikzpicture}[scale=0.5]
		\draw[path] (0,-1) -- (0,0);
		\draw[dotted] (0,0) -- (1,0);
		\end{tikzpicture}}}
\newcommand{\pathlny}{\raisebox{2pt}{\begin{tikzpicture}[scale=0.5]
		\draw[dotted] (0,-1) -- (0,0);
		\draw[thick] (0,0) -- (1,0);
            \draw[path] (-0.05,0)--(0,0);
		\end{tikzpicture}}}
\newcommand{\pathlyy}{\raisebox{2pt}{\begin{tikzpicture}[scale=0.5]
		\draw[thick] (0,-1) -- (0,0);
		\draw[thick] (0,0) -- (1,0);
		\end{tikzpicture}}}
\newcommand{\pathlnn}{\raisebox{2pt}{\begin{tikzpicture}[scale=0.5]
		\draw[dotted] (0,-1) -- (0,0);
		\draw[dotted] (0,0) -- (1,0);
		\end{tikzpicture}}}

\DeclareFontFamily{U}{mathx}{}
\DeclareFontShape{U}{mathx}{m}{n}{<-> mathx10}{}
\DeclareSymbolFont{mathx}{U}{mathx}{m}{n}
\DeclareMathAccent{\widehat}{0}{mathx}{"70}
\DeclareMathAccent{\widecheck}{0}{mathx}{"71}
\def\r{\rightarrow}
\def\u{\uparrow}
\def \PP {\mathbb{P}}
\def \z {\mathbb{Z}}
\def \be {\begin{equation}}
\def \ee {\end{equation}}
\def \de {\delta}
\def \ep {\varepsilon}
\def \D {\mathbf{D}}
\def \E {\mathbf{E}}
\def \k {\kappa}
\def \ll {\langle}
\def \rr {\rangle}
\def \lb {\left(}
\def \rb {\right)}
\def \Du {D^\u}
\def \Dr {D^\r}
\def \Eu {E^\u}
\def \Er {E^\r}

\def \k {\kappa}

\def \se{\mathsf{e}}

\def \th {\theta}
\def \hP {\mathcal{P}}
\def \hQ {\mathcal{Q}}
\DeclareMathOperator{\U}{U}
\def \aa {\mathtt{a}}
\def \bb {\mathtt{b}}
\def \cc {\mathtt{c}}
\def \dd {\mathtt{d}}

\def \AA {\mathtt{A}}
\def \BB {\mathtt{B}}
\def \CC {\mathtt{C}}
\def \DD {\mathtt{D}}
\def \tA{\widetilde{A}}
\def \tB{\widetilde{B}}
\def \tC{\widetilde{C}}
\def \tD{\widetilde{D}}
\def \ttA{\tA\sqrt{t}}
\def \ttB{\tB\sqrt{t}}
\def \ttC{\tC/\sqrt{t}}
\def \ttD{\tD/\sqrt{t}}
\def \sr{\sqrt{r}}
\def \zt{\widetilde{Z_N(t)}}
\def \zn{\widetilde{Z_N(1)}}

\def \pt {\partial_t}
\def \po {|_{t=1}}
\def\qp#1{(#1;q)_\infty}
\def\qps#1{(#1;q)}

\newcommand{\upperRomannumeral}[1]{\uppercase\expandafter{\romannumeral#1}}
\DeclareMathOperator{\dehp}{DEHP}
\DeclareMathOperator{\usw}{USW}

\usepgflibrary{fpu}
\makeatletter
\let\NAT@parse\undefined
\makeatother

\title{Stationary measure for six-vertex model on a strip}
\author{Zongrui Yang}
\address{Department of Mathematics\\
                Columbia University\\
New York, NY, 10027}
\email{zy2417@columbia.edu}
\begin{document}
\begin{abstract}
We study the stochastic six-vertex model on a strip 
\be\label{eq:strip abstract}\left\{(x,y)\in\z^2: 0\leq y\leq x\leq y+N\right\}\ee with two open boundaries. We develop a `matrix product ansatz' method to solve for its stationary measure, based on the compatibility of three types of local moves of any down-right path. The stationary measure on a horizontal path turns out to be a tilting of the stationary measure of the asymmetric simple exclusion process (ASEP) with two open boundaries. 
Similar to open ASEP, the statistics of this tilted stationary measure as the number of sites $N\rightarrow\infty$ (with the bulk and boundary parameters fixed) also exhibit a phase diagram, which is a tilting of the phase diagram of open ASEP. We study the limit of mean particle density as an example. \end{abstract}
\maketitle
\section{Introduction}
Understanding Kardar–Parisi–Zhang (KPZ) universality is one of the major goals in probability and in statistical physics. Much progress has been made through the study of certain `exactly solvable' models in the KPZ universality class. 
The stochastic six-vertex model (S6V) plays a prominent role among these models, since many other models in this class are its specializations (see Figure $1$ and $2$ in \cite{Kuan_Algebraic}), including the asymmetric simple exclusion process (ASEP). 
While most studies focus on KPZ universality on full space, recent progress has been made towards understanding the KPZ equation on a half-line or an open interval, via studies of certain `integrable' models in the KPZ class with one or two open boundaries. These include studies of  half-space stochastic six-vertex models \cite{S6V_in_half_quadrant,Jimmy_He} and polymer models \cite{Half_space_Macdonald,Imamura_Mucciconi_Sasamoto,Jimmy_He,Barraquand_Corwin}.
The stationary measure of open ASEP on an interval was studied by the `matrix product ansatz' method in \cite{DEHP93,USW04,BW17,BW19,BS19,BYW22} which leads to the construction of stationary measure of open KPZ equation on an interval \cite{Corwin_Knizel,BLD21,BKWW21,BK21}.

 Considering the significance and extensive research on open ASEP, it is natural to ask if one can construct and study a six-vertex model with two open boundaries. An open `staggered' six-vertex model is studied in the physics literature \cite{FG1, FG2} (see also \cite{FS}), which is defined on a strip with two boundaries parallel to the $y$-axis and with certain `$K$-matrices' manually placed at the boundaries. A six-vertex model defined on the same strip with a `U-turn' boundary is studied in \cite{BS22, Z}, where the arrows make a U-turn at the boundary and re-enter the system. The boundaries in these models can be seen as of different natures from the single boundary in the half-space six-vertex model \cite{S6V_in_half_quadrant, Jimmy_He}. In this paper, we introduce and study a stochastic six-vertex model on a strip with two open boundaries: $y=x$ and $y=x-N$, where arrows are allowed to enter or exit the system at the open boundaries. This model can be regarded as a natural counterpart to the half-space six-vertex model (which has one open boundary, $y=x$). To the best of our knowledge, the model in this paper has not been introduced before.


The stationary measures of six-vertex models and ASEP have been intensively studied since the 70s. In the full-space case, all the `extremal' stationary measures are classified in \cite{Liggett} for ASEP and in \cite{Lin,Amol_limit shape} for the six-vertex model. They are known as product Bernoulli measures and certain `blocking measures.' For the half-space open ASEP, a certain subset of stationary measures have been studied in \cite{Liggett half-space, G, Sasamoto Williams, BW17}. These measures are known to exhibit a phase diagram involving three phases (see \cite[Figure 3.1]{G}). As mentioned in the first paragraph, the stationary measure of open ASEP on an interval was extensively studied by the `matrix product ansatz' method, which admits a phase diagram (see Figure \ref{fig:phase diagram intro} (a)). In this paper, we will study the stationary measure of the six-vertex model on a strip, using the `matrix product ansatz' method.

 We will provide a detailed definition of the six-vertex model on a strip in subsections \ref{subsec:definition of the model} and \ref{subsection:interacting particle system}. Here, we will only offer a brief introduction. 
Our model is defined on the strip \eqref{eq:strip abstract}
with each edge containing up to one up/right arrow. 
There are initially arrows occupying some outgoing edges of a down-right path $\hP$,
 and we inductively sample through vertex weights:  
 \be\label{bulk_vertex_weights}\begin{gathered}
    \PP\left(\paththick\right) = 1, \quad\PP\left(\pathdotted\right) = 1,\\
    \PP\left(\pathrr\right) = \th_2, \quad\PP\left(\pathru\right) = 1-\th_2,\quad\PP\left(\pathuu\right) = \th_1, \quad\PP\left(\pathur\right) =1-\th_1,
\end{gathered}
\ee
\be\label{left_vertex_weights}
 \PP\left(\pathbrr\hspace{0.2cm}\right) = b,\ \ \PP\left(\pathbuu\hspace{0.2cm}\right) = d,\ \ \PP\left(\pathbru\hspace{0.2cm}\right) = 1-b, \ \ \PP\left(\pathbur\hspace{0.2cm}\right) =1-d,
 \ee
\be\label{right_vertex_weights} 
\PP\left(\pathlyn\hspace{0.2cm}\right)=c,\ \
\PP\left(\pathlny\hspace{0.2cm}\right)=a,\ \ 
\PP\left(\pathlyy\hspace{0.2cm}\right)=1-c,\ \
\PP\left(\pathlnn\hspace{0.2cm}\right)=1-a.\ \ 
\ee  
 where $a,b,c,d$ are boundary parameters and $\th_1,\th_2$ are bulk parameters. See Figure \ref{fig:example6v} for an example of such sampling and Figure \ref{fig:outgoing arrows down-right path} for the outgoing edges of a down-right path. 
We look at the outgoing configurations (i.e. whether the outgoing edges are occupied or not) of all the translated paths $\hP_k=\hP+(k,k)$ for $k\in\mathbb{Z}_{\geq 0}$. 
When we regard $k$ as time then this can be regarded as an interacting particle system whose evolution is governed by the six-vertex model.
There is a standard result (Theorem \ref{thm:scaling limit}) that under a scaling limit, this particle system converges to open ASEP, so it is expected to contain more information than open ASEP. 
We develop a matrix product ansatz method to solve for its stationary measure (Theorem \ref{thm: general matrix ansatz construction}): We begin by prescribing a measure on the outgoing configurations of any down-right path $\hP$ as matrix product states:
$$\mu_{\mathcal{P}}(\tau_1,\dots,\tau_N)=\frac{\ll W|((1-\tau_1)E^{p_1}+\tau_1D^{p_1})\times\dots\times ((1-\tau_N)E^{p_N}+\tau_ND^{p_N})|V\rr}{\ll W|\prod_{i=1}^N(E^{p_i}+D^{p_i})|V\rr},$$
where $p_i\in\left\{\u,\r\right\}$, $1\leq i\leq N$ are outgoing edges of $\hP$ labeled from the up-left of $\hP$ to the down-right of $\hP$, and $\tau_i\in\left\{0,1\right\}$, $1\leq i\leq N$ are occupation variables indicating whether there are arrows on these outgoing edges. 
Then we solve the matrices and vectors $D^\u, D^\r, E^\u, E^\r, \ll W|,|V\rr$ in it using the compatibility of $\mu_\hP$ with three types of local moves  of down-right paths: 
$$
\begin{tikzpicture}[scale=0.6]
		\draw[dotted] (0,0) -- (1,0)--(1,1)--(0,1)--(0,0);
		\draw[very thick] (0,1) -- (0,0) -- (1,0);
		\end{tikzpicture}
  \quad\quad
\raisebox{10pt}{\scalebox{1.5}{$\longmapsto$}}
\quad\quad
\begin{tikzpicture}[scale=0.6]
		\draw[dotted] (0,0) -- (1,0)--(1,1)--(0,1)--(0,0);
		\draw[very thick] (0,1) -- (1,1) -- (1,0);
		\end{tikzpicture},
\quad\quad
\begin{tikzpicture}[scale=0.6]
		\draw[dotted] (0,0) -- (0,1)--(-1,0)--(0,0);
            \draw[very thick] (0,0) -- (-1,0);
		\end{tikzpicture}
  \quad\quad
\raisebox{10pt}{\scalebox{1.5}{$\longmapsto$}}
\quad\quad
\begin{tikzpicture}[scale=0.6]
		\draw[dotted] (0,0) -- (0,1)--(-1,0)--(0,0);
		\draw[very thick] (0,0) -- (0,1);
		\end{tikzpicture}
\quad\quad
\begin{tikzpicture}[scale=0.6]
		\draw[dotted] (0,0) -- (1,0)--(0,-1)--(0,0);
            \draw[very thick] (0,0) -- (0,-1);
		\end{tikzpicture}
  \quad\quad
\raisebox{10pt}{\scalebox{1.5}{$\longmapsto$}}
\quad\quad
\begin{tikzpicture}[scale=0.6]
		\draw[dotted] (0,0) -- (1,0)--(0,-1)--(0,0);
		\draw[very thick] (0,0) -- (1,0);
		\end{tikzpicture}
$$ 
The three sets of compatibility relations \eqref{eq:bulk relations}, \eqref{eq:left boundary relations} and \eqref{eq:right boundary relations} (totally eight relations) coming from local moves look complicated, but they can actually be simplified in Theorem \ref{thm:stationary measure S6V} (after imposing $D^\r=D^\u+I$ and $E^\r=E^\u-I$) to the so-called $\dehp$ algebra 
 $$\D\E-q\E\D=\D+\E,\quad \ll W|(\alpha\E-\gamma\D)=\ll W|,\quad (\beta\D-\de\E)|V\rr=|V\rr.$$ 
The $\dehp$ algebra has appeared in the matrix product ansatz solution of stationary measure of open ASEP in the seminal work \cite{DEHP93} by B. Derrida, M. Evans, V. Hakim and V. Pasquier (see subsection \ref{subsec:matrix ansatz for open ASEP} for a review). In particular, when $\hP$ is a horizontal path, then the stationary measure of six-vertex model on a strip is a tilting of the stationary measure of open ASEP:

\begin{figure} 
\centering
\begin{tikzpicture}[scale=0.6]
\draw[dotted] (-0.5,-0.5)--(4.5,4.5);
\draw[dotted] (4.5,-0.5)--(9.5,4.5);
\draw[dotted] (0,0)--(5,0);
\draw[dotted] (1,1)--(6,1);
\draw[dotted] (2,2)--(7,2);
\draw[dotted] (3,3)--(8,3);
\draw[dotted] (4,4)--(9,4);
\draw[dotted] (0,-0.5)--(0,0);
\draw[dotted] (1,-0.5)--(1,1);
\draw[dotted] (2,-0.5)--(2,2);
\draw[dotted] (3,-0.5)--(3,3);
\draw[dotted] (4,-0.5)--(4,4);
\draw[dotted] (5,0)--(5,4.5);
\draw[dotted] (6,1)--(6,4.5);
\draw[dotted] (7,2)--(7,4.5);
\draw[dotted] (8,3)--(8,4.5);
\draw[dotted] (9,4)--(9,4.5);
\draw[very thick] (2,2)--(2,1)--(3,1)--(3,0)--(5,0);
\draw[very thick] (4,4)--(4,3)--(8,3);
\draw[ultra thick,lgray] (2,2)--(3,2);
\draw[ultra thick,lgray] (3,1)--(3,2);
\draw[ultra thick,lgray] (3,1)--(4,1);
\draw[ultra thick,lgray] (4,0)--(4,1);
\draw[ultra thick,lgray] (5,0)--(5,1);
\draw[ultra thick,lgray] (4,4)--(5,4);
\draw[ultra thick,lgray](5,3)--(5,4);
\draw[ultra thick,lgray](6,3)--(6,4);
\draw[ultra thick,lgray](7,3)--(7,4);
\draw[ultra thick,lgray](8,3)--(8,4);
\node at (4,1) {$\bullet$};
\node at (5,1) {$\bullet$};
\node at (6,1) {$\bullet$};
\node at (3,2) {$\bullet$};
\node at (4,2) {$\bullet$};
\node at (5,2) {$\bullet$};
\node at (6,2) {$\bullet$};
\node at (7,2) {$\bullet$};
\node at (3,3) {$\bullet$};
\node at (4,3) {$\bullet$};
\node at (5,3) {$\bullet$};
\node at (6,3) {$\bullet$};
\node at (7,3) {$\bullet$};
\node at (8,3) {$\bullet$};
\node at (4,4) {$\bullet$};
\end{tikzpicture} 
 
\caption{Outgoing edges on down-right paths and set of vertices $\U(\hP,\hQ)$. The lower thick path is $\hP$ and upper thick path is $\hQ$. The gray edges are outgoing edges of $\hP$ and $\hQ$. Outgoing edges of $\hP$ are labelled from the up-left of the path to the down-right of the path: $p_1=\r$, $p_2=\u$, $p_3=\r$, $p_4=\u$, $p_5=\u$. The thick nodes are vertices in $\U(\hP,\hQ)$. }
\label{fig:outgoing arrows down-right path}
\end{figure}
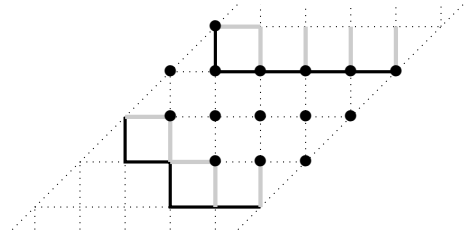

\begin{figure}
\centering
\begin{tikzpicture}[scale=0.6]
\draw[dotted] (-0.5,-0.5)--(4.5,4.5);
\draw[dotted] (4.5,-0.5)--(9.5,4.5);
\draw[dotted] (0,0)--(5,0);
\draw[dotted] (1,1)--(6,1);
\draw[dotted] (2,2)--(7,2);
\draw[dotted] (3,3)--(8,3);
\draw[dotted] (4,4)--(9,4);
\draw[dotted] (0,-0.5)--(0,0);
\draw[dotted] (1,-0.5)--(1,1);
\draw[dotted] (2,-0.5)--(2,2);
\draw[dotted] (3,-0.5)--(3,3);
\draw[dotted] (4,-0.5)--(4,4);
\draw[dotted] (5,0)--(5,4.5);
\draw[dotted] (6,1)--(6,4.5);
\draw[dotted] (7,2)--(7,4.5);
\draw[dotted] (8,3)--(8,4.5);
\draw[dotted] (9,4)--(9,4.5);
\draw[path,thin](2,2)--(5,2)--(5,4);\draw[path,thin](5,3)--(5,4);
\draw[thin](3,1)--(3,2);\draw[thin](5,0)--(5,1); 
\draw[path,thin](8,3)--(8,4);
\draw[path,thin](4,0)--(4,1)--(6,1);
\draw[path,thin](4,0)--(4,1)--(5,1)--(5,4);
\draw[path,thin](2,2)--(6,2)--(6,4);
 \draw[path,thin](2,2)--(3,2)--(3,3)--(4,3)--(4,4)--(5,4); 
\end{tikzpicture} 
\caption{Sample configuration of stochastic six-vertex model on a strip. Down-right paths $\hP$ and $\hQ$ are the same as in Figure \ref{fig:outgoing arrows down-right path} and are omitted. }
\label{fig:example6v}
\end{figure}
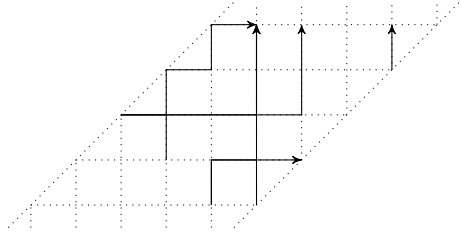

\begin{theorem}\label{thm:main thm tilting}
    Assume that the parameters of six-vertex model on a strip satisfy:
    \be\label{eq:condition Thm 1.1}0<a, b, c, d, \th_1, \th_2<1,\quad \th_1<\th_2,\quad b+d<1.\ee Define:
    \be\label{eq:alphabetagammadelta in intro}(q,\alpha,\beta,\gamma,\de)=\lb\frac{\th_1}{\th_2},\frac{(1-\th_1)a}{\th_2},\frac{(1-\th_2)b}{\th_2(1-b-d)},\frac{(1-\th_2)c}{\th_2},\frac{(1-\th_1)d}{\th_2(1-b-d)}\rb,\ee
and assume $ab/(cd)\notin\{q^l:l=0,1,\dots\}$.
     Let 
$\mu(\tau_1,\dots,\tau_N)$ be the stationary measure of six-vertex model on a strip on a horizontal path (see subsections \ref{subsec:definition of the model} and \ref{subsection:interacting particle system} for its definition). Let
$\pi(\tau_1,\dots,\tau_N)$ be the stationary measure of  open ASEP with particle jump rates $(q,\alpha,\beta,\gamma,\de)$ (see subsection \ref{subsection:interacting particle system} and Figure \ref{fig:openASEP}, where $L=q$ and $R=1$). We have:
\be \label{eq:relation between tau and pi}
\mu(\tau_1,\dots,\tau_N)=r^{\sum_{i=1}^N\tau_i}\pi(\tau_1,\dots,\tau_N)/Z,
\ee
for any $\tau_1,\dots,\tau_N\in\{0,1\}$,
where $r=\frac{1-\th_2}{1-\th_1}\in(0,1)$ and $Z$ is a normalizing constant. 
\end{theorem}

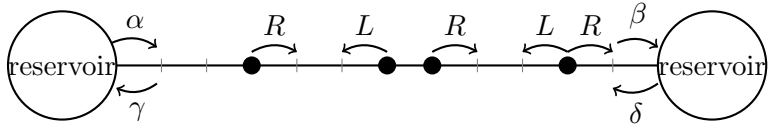
\begin{figure} 
\centering
\begin{tikzpicture}[scale=0.6]
\draw[thick] (-1.2, 0) circle(1.2);
\draw (-1.2,0) node{reservoir};
\draw[thick] (0, 0) -- (12, 0);
\foreach \x in {1, ..., 12} {
	\draw[gray] (\x, 0.15) -- (\x, -0.15);
}
\draw[thick] (13.2,0) circle(1.2);
\draw(13.2,0) node{reservoir};
\fill[thick] (3, 0) circle(0.2);
\fill[thick] (6, 0) circle(0.2);
\fill[thick] (7, 0) circle(0.2);
\fill[thick] (10, 0) circle(0.2);
\draw[thick, ->] (3, 0.3)  to[bend left] node[midway, above]{$R$} (4, 0.3);
\draw[thick, ->] (6, 0.3)  to[bend right] node[midway, above]{$L$} (5, 0.3);
\draw[thick, ->] (7, 0.3) to[bend left] node[midway, above]{$R$} (8, 0.3);
\draw[thick, ->] (10, 0.3) to[bend left] node[midway, above]{$R$} (11, 0.3);
\draw[thick, ->] (10, 0.3) to[bend right] node[midway, above]{$L$} (9, 0.3);
\draw[thick, ->] (-0.1, 0.5) to[bend left] node[midway, above]{$\alpha$} (0.9, 0.4);
\draw[thick, <-] (0, -0.5) to[bend right] node[midway, below]{$\gamma$} (0.9, -0.4);
\draw[thick, ->] (12, -0.4) to[bend left] node[midway, below]{$\de$} (11, -0.5);
\draw[thick, <-] (12, 0.4) to[bend right] node[midway, above]{$\beta$} (11.1, 0.5);
\end{tikzpicture}
\caption{Jump rates in the open ASEP, where we often take $L=q$ and $R=1$.}
\label{fig:openASEP}
\end{figure}

The above theorem will be proved in subsection \ref{subsec:stationary measure of S6V as MPA} as a corollary of the matrix product ansatz (Theorems \ref{thm: general matrix ansatz construction} and \ref{thm:stationary measure S6V}). 
We remark that the above result is surprising to us because it gives a simple relation of stationary measures of two probability systems, and yet we cannot provide a direct probabilistic proof; one must go through the algebraic method of matrix ansatz. 
Under certain special parameter conditions, the general matrix product ansatz (Theorem \ref{thm: general matrix ansatz construction}) also produces the inhomogeneous product Bernoulli and the $q$-volume stationary measures, which will be given in subsection \ref{subsec:stationary measure special cases}. 


We then study the limit of stationary measure of six-vertex model on a horizontal path (the tilted measure \eqref{eq:relation between tau and pi}) as the number of sites $N\rightarrow\infty$, with parameters $a,b,c,d,\th_1,\th_2$ fixed. 
This limit has been well-studied for open ASEP, and it is remarkable that the limits of many statistical quantities under stationary measure exhibit a phase diagram (Figure \ref{fig:phase diagram} (a)) involving only two boundary parameters $A$ and $C$ (see Definition \ref{defn:parameterization open ASEP}), including  mean particle density and density profile \cite{DLS1,SD,USW04} (see also Theorem \ref{thm: open ASEP particle density}), particle current \cite{DEHP93,Sandow,USW04}, correlation functions \cite{ER,UW}, large deviation functionals \cite{DLS1,DLS2,dGE,BW17}, and limit fluctuations around density \cite{DEL,BW19}. See survey papers \cite{D,BE} and more references therein. 
To obtain the phase diagram of six-vertex model on a strip, we study  limits of mean particle density $\rho=\frac{1}{N}\sum_{i=1}^N\tau_i$ under stationary measure \eqref{eq:relation between tau and pi}:

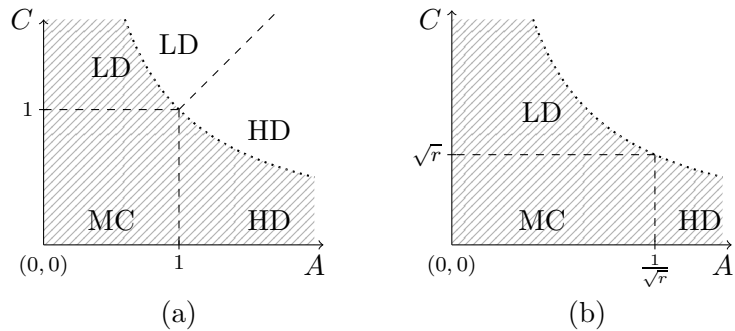
\begin{figure}
\begin{center}
\begin{tikzpicture}[scale=0.6]
\draw[scale = 1,domain=6.8:11,smooth,variable=\x,dotted,thick] plot ({\x},{1/((\x-7)*1/3+2/3)*3+5});
\fill[pattern=north east lines, pattern color=gray!60] (5,5)--(5,10) -- plot [domain=6.8:11]  ({\x},{1/((\x-7)*1/3+2/3)*3+5}) -- (11,5) -- cycle;
 \draw[->] (5,5) to (5,10);
 \draw[->] (5.,5) to (11.2,5);
   \draw[-, dashed] (5,8) to (8,8);
   \draw[-, dashed] (8,8) to (8,5);
   \draw[-, dashed] (8,8) to (10.2,10.2);
   \node [left] at (5,8) {\scriptsize$1$};
   \node[below] at (8,5) {\scriptsize $1$};
     \node [below] at (11,5) {$A$};
   \node [left] at (5,10) {$C$};
  \draw[-] (8,4.9) to (8,5.1);
   \draw[-] (4.9,8) to (5.1,8);
 \node [below] at (5,5) {\scriptsize$(0,0)$};
    \node [above] at (6.5,8.5) {LD};\node [above] at (8,9) {LD};
    \node [below] at (10,6) {HD};  \node [below] at (10,8) {HD}; 
 \node [below] at (6.5,6) {MC};
  \node[below] at (8,4){(a)};
\end{tikzpicture}
\quad\quad
\begin{tikzpicture}[scale=0.6]
\draw[scale = 1,domain=6.8:11,smooth,variable=\x,dotted,thick] plot ({\x},{1/((\x-7)*1/3+2/3)*3+5});
\fill[pattern=north east lines, pattern color=gray!60] (5,5)--(5,10) -- plot [domain=6.8:11]  ({\x},{1/((\x-7)*1/3+2/3)*3+5}) -- (11,5) -- cycle;
 \draw[->] (5,5) to (5,10.2);
 \draw[->] (5.,5) to (11.2,5);
   \draw[-, dashed] (9.5,7) to (5,7);
   \draw[-, dashed] (9.5,7) to (9.5,5);
   \node [left] at (5,7) {\scriptsize$\sqrt{r}$};
   \node[below] at (9.5,5) {\scriptsize $\frac{1}{\sqrt{r}}$};
     \node [below] at (11,5) {$A$};
   \node [left] at (5,10) {$C$};
  \draw[-] (9.5,4.9) to (9.5,5.1);
   \draw[-] (4.9,7) to (5.1,7);
 \node [below] at (5,5) {\scriptsize$(0,0)$};
    \node [above] at (7,7.5) {LD}; 
    \node [below] at (10.5,6) {HD}; 
 \node [below] at (7,6) {MC};
 \node[below] at (8,4){(b)};
\end{tikzpicture}
\end{center}
\caption{(a) Phase diagram for open ASEP. (b) Phase diagram for six-vertex model on a strip.  } 
\label{fig:phase diagram intro}
\end{figure}

 \begin{theorem}\label{thm:main theorem section 3}
Consider the six-vertex model on a strip with bulk parameters $\th_1, \th_2$ and boundary parameters $a, b, c, d$. Assume that the parameters satisfy \eqref{eq:condition Thm 1.1}. 
We will use an alternative parameterization of the system by $(q,r,A,B,C,D)$, where $q=\th_1/\th_2$, $r=(1-\th_2)/(1-\th_1)$, and $A,B,C,D$ are defined in terms of $(q,\alpha,\beta,\gamma,\delta)$ \eqref{eq:alphabetagammadelta in intro} by Definition \ref{defn:parameterization open ASEP}.
Then on the fan region $AC<1$, under the technical condition
\be\label{eq:technical condition}
|A\sqrt{r}|,|B\sqrt{r}|,|C/\sqrt{r}|,|D/\sqrt{r}|\neq q^{-l} \quad\text{ for }  l=0,1,2,\dots
\ee 
the limits of mean particle density are given by:
\be\label{eq:limit of mean density S6V}
\lim_{N\rightarrow\infty}\mathbb{E}_\mu\rho=
\begin{cases}
\frac{\sqrt{r}}{1+\sqrt{r}}, & A\leq 1/\sqrt{r}, C\leq\sqrt{r} \text{ (maximal current phase with boundary)},\\
\frac{Ar}{1+Ar}, & A> 1/\sqrt{r}, AC<1 \text{ (high density phase)},\\
\frac{r}{r+C}, & C>\sqrt{r}, AC<1 \text{ (low density phase)}.
\end{cases}
\ee
When we also assume $a+c<1$ we do not need the technical condition \eqref{eq:technical condition}.
\end{theorem}

The above theorem will be proved in subsection \ref{subsec:proof of limit of mean density}. In the proof we will utilize an auxiliary Markov process known as the Askey-Wilson process, which is developed in the literature on open ASEP stationary measure \cite{BW10,BW17,USW04} (see a brief introduction in subsections \ref{subsec:background AW process} and \ref{subsec:background asymptotics open ASEP}).  Theorem \ref{thm:main theorem section 3} provides the phase diagram (Figure \ref{fig:phase diagram intro} (b)) of six-vertex model on a strip, which is a tilting of phase diagram of open ASEP (Figure \ref{fig:phase diagram intro} (a)). 
The shadowed regions in Figure \ref{fig:phase diagram intro} correspond to the fan regions $AC<1$ in the phase diagrams. At present, it is uncertain whether the phase diagram in Figure \ref{fig:phase diagram intro} (b) can be extended to include the shock region. Nevertheless, a recent work  \cite{WWY} may offer the necessary techniques (see Remark \ref{remark:fan region and extension}). We defer this aspect to future research.

There remains many open questions to investigate, and we list a few of them. 
We expect the density profile and limit fluctuations of stationary measure \eqref{eq:relation between tau and pi} can be studied following techniques in \cite{BW17,BW19}. 
An interesting feature of the model in this paper is that it is a six-vertex model but studied by matrix ansatz coming from open ASEP, and one can ask if the techniques in previous works of six-vertex model (e.g. Bethe ansatz, symmetric functions, etc.) can also be applied.
Moreover, since the open ASEP can be seen as a specialization of the six-vertex model on a strip (Theorem \ref{thm:scaling limit} and \eqref{eq:relation between tau and pi}), we expect that our model is more flexible in taking scaling limits. In particular, since a scaling limit of open ASEP converges to the open KPZ equation \cite{CS,P}, one expect similar scaling limits of our model.  

\section*{Acknowledgement}
The author heartily thank his advisor Ivan Corwin for suggesting this problem, for generous and helpful discussions during this project, and for offering advice on paper writing.
We thank Amol Aggarwal, Jeffrey Kuan and Jimmy He for helpful discussions and Chenyang Zhong for writing a computer simulation. 
The author was supported by Ivan Corwin’s NSF grant DMS-1811143 as well as the Fernholz Foundation’s “Summer Minerva Fellows” program.

\section{The six-vertex model on a strip and stationary measure}\label{section 2}
\subsection{Definition of the model}\label{subsec:definition of the model}
We consider certain configurations of arrows on edges of the strip:
\begin{equation}\label{eq:strip}
\left\{(x,y)\in\z^2: 0\leq y\leq x\leq y+N\right\},
\end{equation}
where each edge can contain up to one up/right arrow. We refer the vertices $(y,y)$ as left boundary vertices, $(y+N,y)$ as right boundary vertices and all other vertices of the strip as bulk vertices.  
The left and/or bottom edges of each vertex are referred to as its incoming edges, and the right and/or top edges are called its outgoing edges.

We will use the word `down-right path' to mean a path $\mathcal{P}$ that goes from a vertex on the left boundary to a vertex on the right boundary, with each step going downwards or rightwards by $1$. 
Observe that each down-right path on the strip has length $N$ and  there are exactly $N$ outgoing up/right edges on the path. Each outgoing edge can be occupied by up to one arrow, which gives $2^N$ many `outgoing configurations' of $\hP$.  
Suppose $\mathcal{Q}$ is any down-right path that sits above $\hP$, which may contain edges coinciding with $\hP$. We denote by $\U(\hP,\hQ)$ the set of vertices between $\hP$ and $\hQ$, including those on $\hQ$ but excluding those on $\hP$. See Figure \ref{fig:outgoing arrows down-right path} for an illustration.
 
 Suppose we are given a (deterministic) outgoing configuration of $\hP$. We define inductively a Markovian sampling procedure to generate configurations. Suppose we are at the vertex $(x,y)\in\U(\hP,\hQ)$ and we have sampled through all vertices $(x',y')\in\U(\hP,\hQ)$ such that either $y'<y$ or $y'=y$ and $x'<x$. Then for each incoming edge of $(x,y)$ we have already assigned an arrow or no arrow to it. We sample the outgoing edges of $(x,y)$ according to the probabilities  \eqref{bulk_vertex_weights}, \eqref{left_vertex_weights} and \eqref{right_vertex_weights} respectively, in the cases when $(x,y)$ is a bulk/right boundary/left boundary vertex.
We sequentially sample through all vertices in $\U(\hP,\hQ)$ and obtain a probability measure on the set of all outgoing configurations of $\hQ$.
See Figure \ref{fig:example6v} for an example of this sampling.

\subsection{Interacting particle systems}\label{subsection:interacting particle system}
For a down-right path $\hP$ on the strip, we label its outgoing edges from the up-left of $\hP$ to the down-right of $\hP$ by $p_1,\dots,p_N\in\left\{\u,\r\right\}$ (see Figure \ref{fig:outgoing arrows down-right path}), where $\u$ denotes a vertical edge and $\r$ denotes a horizontal edge (these should not be confused with arrows in the six-vertex model). The $2^N$ `outgoing configurations' of $\hP$ can be encoded in occupation variables $\tau=(\tau_1,\dots,\tau_N)\in\left\{0,1\right\}^N$, where $\tau_i$ indicates whether or not the edge $p_i$ is occupied. Assume $\mathcal{Q}$ is a down-right path that sits above $\hP$, then the sampling procedure in subsection \ref{subsec:definition of the model} can be encoded as a probability transition matrix $P_{\hP,\mathcal{Q}}(\tau,\tau')$, where $\tau,\tau'\in\left\{0,1\right\}^N$ are occupation variables of outgoing edges of $\hP$ and $\hQ$.

\begin{definition}\label{defn:interacting particle system down-right path}
Assume $\hP$ is a down-right path on the strip. Denote by $\hP_k$ the up-right translation of $\hP$ by $(k,k)$, for $k\in\mathbb{Z}_{\geq 0}$.  
We consider a time-homogeneous Markov chain $(\tau(k))_{k\geq 0}$ with some initial outgoing configuration $\tau(0)\in\left\{0,1\right\}^N$ of $\hP$ and the same transition probability matrix $P_{\hP_k,\hP_{k+1}}(\tau,\tau')=P_{\hP,\hP_{1}}(\tau,\tau')$ in each step. 
We regard this Markov chain as a particle system on the lattice $\left\{1,\dots,N\right\}$ where a particle sits at position $i$ at time $k$ if and only if $\tau_i(k)=1$, and we refer it as the interacting particle system of the six-vertex model on a strip on down-right path $\hP$. 
\end{definition}
\begin{remark}
    The sequential update rules of this particle system can be written in a similar way as in the full space case in \cite[subsection 2.2]{BCG}, however more complicated since there are two open boundaries. We will not write the specific update rules since we will not use them.
\end{remark}

We will compare the particle system in Definition \ref{defn:interacting particle system down-right path} to open asymmetric simple exclusion process (ASEP). 
The open ASEP is a continuous-time interacting particle system on the lattice $\left\{1,\dots,N\right\}$, where each site can contain up to one particle. Particles are allowed to move to its nearest left/right neighbor and can also enter or exit the system at two boundary sites $1, N$. Specifically, particles move to the left with rate $L$ and to the right with rate $R$, but a move is prohibited (excluded) if the target site is already occupied. Particles enter the system and are placed at site $1$ with rate $\alpha$ and at site $N$ with rate $\delta$, provided that the site is empty. Particles are also removed with rate $\gamma$ from site $1$ and with rate $\beta$ from site $N$. These jump rates are summarized in Figure \ref{fig:openASEP}.

The following shows that under a scaling limit, the six-vertex model converges to the continuous-time open ASEP. Similar results in the full and half space settings are already obtained in \cite{Amol_Convergence,S6V_in_half_quadrant,Jimmy_He}.
Our situation is simpler than in those settings since there are finitely many sites.
\begin{theorem}\label{thm:scaling limit}
Consider the interacting particle system of six-vertex model on a strip (with parameters $a,b,c,d,\th_1,\th_2$) on a down-right path $\hP$. Scale bulk and boundary parameters $(a,b,c,d,\th_1,\th_2)=(\alpha\ep,\beta\ep,\gamma\ep,\de\ep,L\ep,R\ep)$ and scale time $\eta=[\ep^{-1}t]$. Take $\ep\rightarrow 0$, then the system at time $\eta$ converges weakly to the continuous-time open ASEP at time $t$, on the same lattice $\left\{1,\dots,N\right\}$ with particle jump rates $(\alpha,\beta,\gamma,\de,L,R)$ (see Figure \ref{fig:openASEP}) and with the same initial condition.
\end{theorem}
\begin{proof}
    When $\ep\rightarrow 0$ the arrows in the six-vertex model will essentially always wiggle, i.e. alternate between going up and going right, and will almost never keep going the same direction in two subsequent steps. Hence in each step the interacting particle system will stay put at a probability near $1$.
    Since we are scaling time $\eta=[\ep^{-1}t]$, we want to keep track of all the possible changes to the system that can happen in one step with a probability of order $O(\ep)$. 
    Observe that there are $N$ up-right zig-zag paths, and any down-right path $\hP$ must have exactly one outgoing edge on each of these paths. The arrows will evolve along its own zig-zag path at a probability near $1$, however it will change to its left/right neighboring zig-zag paths when it continues going the same up/right direction in two subsequent steps. At the open boundaries an arrow can eject out of the system or enter the system and then goes along the leftmost/rightmost zig-zag path. 
    By the vertex weights  \eqref{bulk_vertex_weights}, \eqref{left_vertex_weights} and \eqref{right_vertex_weights} we can see that
    exactly one of the following can happen in one step with a probability of order $O(\ep)$:
    \begin{enumerate}
    \item [$\bullet$] If $1$ is unoccupied, a particle can enter the system and placed at $1$ with probability 
    $\alpha\ep+O(\ep^2)$.
    \item [$\bullet$] If $1$ is occupied by a particle, it can be ejected out of the system with probability
    $\gamma\ep+O(\ep^2)$.
    \item [$\bullet$] One of the particles that is already in the system can either jump $1$ step left with probability $L\ep+O(\ep^2)$, or $1$ step right with probability $R\ep+O(\ep^2)$, if not blocked by other particle. Particle at $1$ can only jump right and particle at $N$ can only jump left.
    \item [$\bullet$] If $N$ is unoccupied, a particle can enter the system and placed at $N$ with probability $\de\ep+O(\ep^2)$.
    \item [$\bullet$] If $N$ is occupied by a particle, it can be ejected out of the system with probability $\beta\ep+O(\ep^2)$.
\end{enumerate}
    Choose a basis of the state space $\left\{0,1\right\}^N$.
    Denote by $A_{\ep}=P_{\hP,\hP_1}(\tau,\tau')$ the $2^N\times 2^N$ transition matrix of the interacting particle system related to six-vertex model on the down-right path $\hP$, with parameters $(a,b,c,d,\th_1,\th_2)=(\alpha\ep,\beta\ep,\gamma\ep,\de\ep,L\ep,R\ep)$. Denote by $Q$ the infinitesimal generator ($Q$-matrix) of the open ASEP with jump rates $(\alpha,\beta,\gamma,\delta,L,R)$.
    The observation above tells us $A_{\ep}=I+\ep Q+O(\ep^2)$.
    Hence $\lim_{\ep\rightarrow 0}A_{\ep}^{[\ep^{-1}t]}=e^{tQ}$. 
    The left hand side is the transition probability for the (discrete-time) particle system with parameters $(\alpha\ep,\beta\ep,\gamma\ep,\de\ep,L\ep,R\ep)$ from time $0$ to time $\eta=[\ep^{-1}t]$, and the right hand side is the transition probability for the continuous-time open ASEP from time $0$ to time $t$. Since the initial data are the same we conclude the weak convergence.
\end{proof}

\subsection{Matrix product ansatz of open ASEP}\label{subsec:matrix ansatz for open ASEP}
In this subsection we recall the matrix product ansatz solution of the stationary measure of open ASEP first developed in the seminal work \cite{DEHP93} by B. Derrida, M. Evans, V. Hakim and V. Pasquier. See also \cite{Corwin_survey} for a nice survey.

We start with open ASEP on the lattice $\left\{1,\dots,N\right\}$ with particle jump rates given in Figure \ref{fig:openASEP}. We will always assume $L=q$, $R=1$ and 
\be\label{eq:conditions open ASEP}
\alpha,\beta>0,\quad \gamma,\de\geq 0,\quad 0\leq q<1.
\ee 
Under these assumptions the open ASEP is irreducible as a Markov process on the finite state space $\left\{0,1\right\}^N$. We denote by $\pi=\pi(\tau_1,\dots,\tau_N)$ its (unique) stationary measure, where $\tau_1,\dots,\tau_N\in\left\{0,1\right\}$ are occupation variables of $N$ sites.
It is known since \cite{DEHP93} that the stationary measure $\pi$ can be written as the following matrix product:
\begin{theorem}[\cite{DEHP93}]\label{thm:matrix ansatz from DEHP}
 Assume \eqref{eq:conditions open ASEP}. 
 Suppose that there are matrices $\D$, $\E$, a row vector $\ll W|$ and a column vector $|V\rr$ with the same (possibly infinite) dimension, satisfying:
\be \label{eq:DEHP algebra}
     \D\E-q\E\D=\D+\E,\quad \ll W|(\alpha\E-\gamma\D)=\ll W|,\quad (\beta\D-\de\E)|V\rr=|V\rr, 
    \ee 
(which is commonly referred to as the DEHP algebra). Then for any $t_1,\dots,t_N>0$,  
\be \label{eq:MPA open ASEP}
    \mathbb{E}_{\pi}\lb\prod_{j=1}^Nt_j^{\tau_j}\rb=\frac{\ll W|(\E+t_1\D)\times\dots\times(\E+t_N\D)|V\rr}{\ll W|(\E+\D)^N|V\rr},
    \ee 
assuming that the denominator $\ll W|(\E+\D)^N|V\rr$ is nonzero. 

\end{theorem}

 We refer the reader to \cite{DEHP93} for the proof of this theorem. 
\begin{remark}\label{remark:associativity}
    Here we implicitly assume that all the admissible finite products of the matrices and vectors $\D$, $\E$, $\ll W|$ and $|V\rr$ are well-defined (i.e. convergent if they are infinite dimensional) and satisfy the associativity property. These properties will also be implicitly assumed in the matrix ansatz in Theorem \ref{thm: general matrix ansatz construction}.  
   One can observe that the $\usw$ representation \cite{USW04} of the DEHP algebra (see Remark \ref{remark:USW}) satisfy these properties, since $\D$ and $\E$ are tridiagonal matrices and $\ll W|$ and $|V\rr$ are finitely supported vectors. 

\end{remark}  
 \begin{remark}\label{remark:nonsingular}
It has been noted in \cite{ER,MS} that the matrix ansatz \eqref{eq:MPA open ASEP} possibly does not work (i.e. the denominator may equal to zero) when $\alpha\beta=q^l\gamma\delta$ for some $l=0,1,\dots$. 
They are referred to as the `singular' cases of the matrix ansatz, for which an alternative method is developed in \cite{BS19}.
In this paper we only consider the `non-singular' case
\be \label{eq:singular case}
\alpha\beta\neq q^l\gamma\delta\quad\text{for any }l=0,1,\dots.
\ee 
Assume \eqref{eq:conditions open ASEP}, \eqref{eq:singular case} and $\ll W|V\rr>0$, then it can be shown that the matrix products \be\label{eq:matrix products}\ll W|\D^{n_1}\E^{m_1}\dots\D^{n_k}\E^{m_k}|V\rr\ee are strictly positive, for any $k,n_1,m_1,\dots,n_k,m_k\geq 0$. In particular, the denominator $\ll W|(\E+\D)^N|V\rr$ of \eqref{eq:MPA open ASEP} is strictly positive. 
Moreover, the matrix products \eqref{eq:matrix products} only depend on the value of $\ll W|V\rr$, parameters $q,\alpha,\beta,\gamma,\delta$ and numbers $k,n_1,m_1,\dots,n_k,m_k$, and are independent on the specific choices of $\D,\E,\ll W|,|V\rr$ satisfying the $\dehp$ algebra \eqref{eq:DEHP algebra}. See \cite[Appendix A]{MS} for a proof of these facts. 
\end{remark}

\begin{remark}\label{remark:USW}
    It is a highly nontrivial task to find concrete examples of $\D$, $\E$, $\ll W|$ and $|V\rr$ that satisfy the $\dehp$ algebra \eqref{eq:DEHP algebra}. For general parameters $q,\alpha,\beta,\gamma,\delta$, an example was found in the seminal work \cite{USW04} by  M. Uchiyama, T. Sasamoto and M. Wadati, which is often referred to as the $\usw$ representation of the DEHP algebra.  In such an example, $\D$ and $\E$ are infinite tridiagonal matrices with entries closely related to the Jacobi matrices of the Askey-Wilson orthogonal polynomials, $\ll W|=(1,0,0,\dots)$ and $|V\rr=(1,0,0,\dots)^T$. 
\end{remark}

\subsection{Stationary measure of six-vertex model on a strip}
\label{subsec:stationary measure of S6V as MPA}
Consider the six-vertex model on a strip with parameters $a,b,c,d,\th_1,\th_2$.
Assume $\hP$ is a down-right path with outgoing edges $p_1,\dots,p_N\in\left\{\u,\r\right\}$, which defines an interacting particle system as in Definition \ref{defn:interacting particle system down-right path}. In this subsection we will always assume
$a, b, c, d, \th_1, \th_2\in(0,1)$
so that this system is irreducible as a Markov process on the finite state space $\left\{0,1\right\}^N$. We denote its unique stationary measure by $\mu_{\mathcal{P}}=\mu_{\mathcal{P}}(\tau_1,\dots,\tau_N)$, where $\tau_1,\dots,\tau_N\in\left\{0,1\right\}$ are occupation variables of $N$ sites. We develop a matrix product ansatz method based on local moves \eqref{eq:local move bulk}, \eqref{eq:local move left boundary} and \eqref{eq:local move right boundary} of down-right paths to solve for the stationary measure $\mu_{\mathcal{P}}$. An interesting feature is that this matrix ansatz ties together the interacting particle systems arising from different down-right paths.
 We then realize in Theorem \ref{thm:stationary measure S6V} that the eight compatibility relations of the matrix ansatz (which arise from local moves) can be reduced to the $\dehp$ algebra. As a corollary, in the special case when $\hP$ is a horizontal path, we prove Theorem \ref{thm:main thm tilting}. 

\begin{theorem}\label{thm: general matrix ansatz construction}
    Suppose that there are 
     matrices $D^\u, D^\r, E^\u, E^\r$, row vector $\ll W|$ and column vector $|V\rr$ with the same (possibly infinite) dimension satisfying the following three sets of relations:
\be\label{eq:bulk relations}\begin{gathered}
        D^\u D^\r=D^\r D^\u,\quad E^\u E^\r=E^\r E^\u, \\
       D^\u E^\r=(1-\th_2)D^\r E^\u+\th_1E^\r D^\u, \\
        E^\u D^\r=\th_2D^\r E^\u+(1-\th_1)E^\r D^\u,
\end{gathered}\ee
\be\label{eq:left boundary relations}\begin{gathered}
        \ll W|D^\r=(1-c)\ll W|D^\u+a\ll W|E^\u, \\
       \ll W|E^\r=(1-a)\ll W|E^\u+c\ll W|D^\u,
\end{gathered}\ee
\be\label{eq:right boundary relations}\begin{gathered}
        D^\u|V\rr=(1-b)D^\r|V\rr+dE^\r|V\rr, \\
       E^\u|V\rr=(1-d)E^\r|V\rr+bD^\r|V\rr.
\end{gathered}\ee
 The stationary measure of six-vertex model on a  
strip on down-right path $\hP$  is given by:  
 \be\label{eq: matrix ansatz down-right path}
\mu_{\mathcal{P}}(\tau_1,\dots,\tau_N)=\frac{\ll W|((1-\tau_1)E^{p_1}+\tau_1D^{p_1})\times\dots\times ((1-\tau_N)E^{p_N}+\tau_ND^{p_N})|V\rr}{\ll W|\prod_{i=1}^N(E^{p_i}+D^{p_i})|V\rr},
\ee  
 where $p_1,\dots,p_N\in\left\{\u,\r\right\}$ are the outgoing edges of $\hP$ and $\tau_1,\dots,\tau_N\in\left\{0,1\right\}$ are occupation variables on these edges.
We assume the denominator of \eqref{eq: matrix ansatz down-right path} is nonzero. 
\end{theorem} 
\begin{proof}
Consider the collection of signed measures $\mu_{\hP}$ (with total mass $1$) indexed by down-right paths $\hP$ given by the matrix product states \eqref{eq: matrix ansatz down-right path}. We prove the following:
    \begin{claim}
    The collection $\mu_\hP$ of signed measures for down-right paths $\hP$ is compatible with the evolution of six-vertex model, i.e. for any path $\hQ$ sitting above $\hP$ (which may have coinciding edges),  
    \be\label{eq:transition probability}
    \sum_\tau P_{\hP,\mathcal{Q}}(\tau,\tau')\mu_\hP(\tau)=\mu_{\mathcal{Q}}(\tau'),\quad\forall\tau,\tau'\in\left\{0,1\right\}^N.
    \ee
    \end{claim}
    
    Observe that if we consider the translated path $\hP_1=\hP+(1,1)$, then $\mu_{\hP_1}$ is the same as $\mu_\hP$ as signed measures on $\left\{0,1\right\}^N$, since the outgoing edges of $\hP_1$ are also $p_1,\dots,p_N\in\left\{\u,\r\right\}$ (so that the elements that we put in the matrix product states \eqref{eq: matrix ansatz down-right path} are the same).
    Suppose the above claim holds, we can take $\mathcal{Q}=\hP_1$ and hence $\mu_\hP$ is an eigenvector with eigenvalue $1$ of the transition matrix $P_{\hP,\hP_1}(\tau,\tau')$ of the (irreducible) interacting particle system defined by $\hP$. By Perron-Frobenius theorem $\mu_\hP$ is the unique stationary probability measure of this system. 
    
    We introduce three types of `local moves' of a down-right path, where the thick paths denote locally the down-right path:
\be \label{eq:local move bulk}
\begin{tikzpicture}[scale=0.6]
		\draw[dotted] (0,0) -- (1,0)--(1,1)--(0,1)--(0,0);
		\draw[very thick] (0,1) -- (0,0) -- (1,0);
		\end{tikzpicture}
  \quad\quad
\raisebox{10pt}{\scalebox{1.5}{$\longmapsto$}}
\quad\quad
\begin{tikzpicture}[scale=0.6]
		\draw[dotted] (0,0) -- (1,0)--(1,1)--(0,1)--(0,0);
		\draw[very thick] (0,1) -- (1,1) -- (1,0);
		\end{tikzpicture},
\ee 
\be \label{eq:local move left boundary}
\begin{tikzpicture}[scale=0.6]
		\draw[dotted] (0,0) -- (0,1)--(-1,0)--(0,0);
            \draw[very thick] (0,0) -- (-1,0);
		\end{tikzpicture}
  \quad\quad
\raisebox{10pt}{\scalebox{1.5}{$\longmapsto$}}
\quad\quad
\begin{tikzpicture}[scale=0.6]
		\draw[dotted] (0,0) -- (0,1)--(-1,0)--(0,0);
		\draw[very thick] (0,0) -- (0,1);
		\end{tikzpicture}
\ee 
\be \label{eq:local move right boundary}
\begin{tikzpicture}[scale=0.6]
		\draw[dotted] (0,0) -- (1,0)--(0,-1)--(0,0);
            \draw[very thick] (0,0) -- (0,-1);
		\end{tikzpicture}
  \quad\quad
\raisebox{10pt}{\scalebox{1.5}{$\longmapsto$}}
\quad\quad
\begin{tikzpicture}[scale=0.6]
		\draw[dotted] (0,0) -- (1,0)--(0,-1)--(0,0);
		\draw[very thick] (0,0) -- (1,0);
		\end{tikzpicture}
\ee

We remark that by sequentially performing these local moves, a down-right path $\mathcal{P}$ can be updated to any down-right path $\mathcal{Q}$ sitting above it. See Figure \ref{fig:translation path} for an example of achieving an upper translation of a horizontal path by these local moves. Therefore \eqref{eq:transition probability} can be guaranteed by its special case when $\mathcal{Q}$ is  a local move $\widetilde{\hP}$ of $\hP$:
\be\label{eq:transition probability locally}
    \sum_\tau P_{\hP,\widetilde{\hP}}(\tau,\tau')\mu_\hP(\tau)=\mu_{\widetilde{\hP}}(\tau'),\quad\forall\tau,\tau'\in\left\{0,1\right\}^N.
\ee

As we put matrix product states \eqref{eq: matrix ansatz down-right path} of $\mu_\hP$ and $\mu_{\widetilde{\hP}}$ into \eqref{eq:transition probability locally}, all of the terms coincide except two that went through the local move in the bulk, or one that went through local move at the left/right boundary. As a sufficient condition for \eqref{eq:transition probability locally} to hold, we only need to keep track of the updated terms.
In the following diagrams the thick paths represent locally the down-right paths $\hP$ and $\widetilde{\hP}$, and thin arrows denote locally the outgoing configurations.

\begin{figure}
\centering
\begin{tikzpicture}[scale=0.6]
		\draw[dotted] (-0.5,-0.5)--(1.5,1.5);
            \draw[dotted] (1.5,-0.5)--(3.5,1.5);
            \draw[dotted] (0,0)--(2,0);
            \draw[dotted] (1,1)--(3,1);
            \draw[dotted] (0,-0.5)--(0,0);
            \draw[dotted] (1,-0.5)--(1,1);
            \draw[dotted] (2,0)--(2,1.5);
            \draw[dotted] (3,1)--(3,1.5);
		\draw[very thick] (0,0) -- (2,0);
		\end{tikzpicture}
\raisebox{15pt}{\scalebox{1.5}{$\longmapsto$}}\quad
\begin{tikzpicture}[scale=0.6]
		\draw[dotted] (-0.5,-0.5)--(1.5,1.5);
            \draw[dotted] (1.5,-0.5)--(3.5,1.5);
            \draw[dotted] (0,0)--(2,0);
            \draw[dotted] (1,1)--(3,1);
            \draw[dotted] (0,-0.5)--(0,0);
            \draw[dotted] (1,-0.5)--(1,1);
            \draw[dotted] (2,0)--(2,1.5);
            \draw[dotted] (3,1)--(3,1.5);
		\draw[very thick] (1,1) -- (1,0)--(2,0);
		\end{tikzpicture}
\raisebox{15pt}{\scalebox{1.5}{$\longmapsto$}}\quad 
\begin{tikzpicture}[scale=0.6]
		\draw[dotted] (-0.5,-0.5)--(1.5,1.5);
            \draw[dotted] (1.5,-0.5)--(3.5,1.5);
            \draw[dotted] (0,0)--(2,0);
            \draw[dotted] (1,1)--(3,1);
            \draw[dotted] (0,-0.5)--(0,0);
            \draw[dotted] (1,-0.5)--(1,1);
            \draw[dotted] (2,0)--(2,1.5);
            \draw[dotted] (3,1)--(3,1.5);
		\draw[very thick] (1,1) -- (2,1)--(2,0);
		\end{tikzpicture}
\raisebox{15pt}{\scalebox{1.5}{$\longmapsto$}}\quad 
\begin{tikzpicture}[scale=0.6]
		\draw[dotted] (-0.5,-0.5)--(1.5,1.5);
            \draw[dotted] (1.5,-0.5)--(3.5,1.5);
            \draw[dotted] (0,0)--(2,0);
            \draw[dotted] (1,1)--(3,1);
            \draw[dotted] (0,-0.5)--(0,0);
            \draw[dotted] (1,-0.5)--(1,1);
            \draw[dotted] (2,0)--(2,1.5);
            \draw[dotted] (3,1)--(3,1.5);
		\draw[very thick] (1,1) -- (3,1);
		\end{tikzpicture}
\caption{Upper translation of a horizontal path via three local moves.}
\label{fig:translation path}
\end{figure}
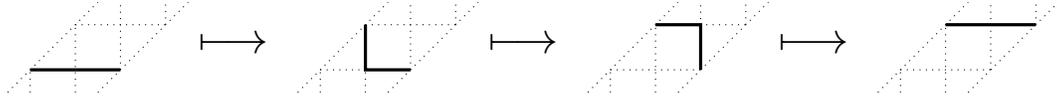
We first consider a bulk local move \eqref{eq:local move bulk}. 
The outgoing edges of $\hP$ and $\widetilde{\hP}$ coincide except for those two edges that went through the local move.  
The following are local terms in the matrix ansatz of $\mu_{\hP}$ on the possible local configurations:
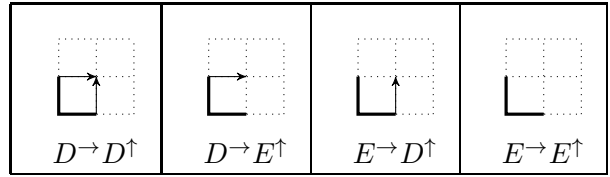
\begin{align*}
\begin{tabular}{|c|c|c|c|} 
\hline 
& & &
\\ 
\quad
\raisebox{-4pt}{\begin{tikzpicture}[scale=0.5]
		\draw[dotted] (0,0)--(2,0);\draw[dotted] (0,1)--(2,1);\draw[dotted] (0,2)--(2,2);
            \draw[dotted] (0,0)--(0,2);\draw[dotted] (1,0)--(1,2);\draw[dotted] (2,0)--(2,2);
		\draw[very thick] (0,1) -- (0,0) -- (1,0);
            \draw[path,thin] (0,1)--(1,1); \draw[path,thin] (1,0)--(1,1);
		\end{tikzpicture}}
\quad
&
\quad
\raisebox{-4pt}{\begin{tikzpicture}[scale=0.5]
		\draw[dotted] (0,0)--(2,0);\draw[dotted] (0,1)--(2,1);\draw[dotted] (0,2)--(2,2);
            \draw[dotted] (0,0)--(0,2);\draw[dotted] (1,0)--(1,2);\draw[dotted] (2,0)--(2,2);
		\draw[very thick] (0,1) -- (0,0) -- (1,0);
            \draw[path,thin] (0,1)--(1,1); 
		\end{tikzpicture}}
\quad
&
\quad
\raisebox{-4pt}{\begin{tikzpicture}[scale=0.5]
		\draw[dotted] (0,0)--(2,0);\draw[dotted] (0,1)--(2,1);\draw[dotted] (0,2)--(2,2);
            \draw[dotted] (0,0)--(0,2);\draw[dotted] (1,0)--(1,2);\draw[dotted] (2,0)--(2,2);
		\draw[very thick] (0,1) -- (0,0) -- (1,0);
  \draw[path,thin] (1,0)--(1,1);
		\end{tikzpicture}}
\quad
&
\quad
\raisebox{-4pt}{\begin{tikzpicture}[scale=0.5]
		\draw[dotted] (0,0)--(2,0);\draw[dotted] (0,1)--(2,1);\draw[dotted] (0,2)--(2,2);
            \draw[dotted] (0,0)--(0,2);\draw[dotted] (1,0)--(1,2);\draw[dotted] (2,0)--(2,2);
		\draw[very thick] (0,1) -- (0,0) -- (1,0);
		\end{tikzpicture}}
  \quad
\\[0.3cm]
\quad
$\Dr\Du$
\quad
& 
\quad
$\Dr\Eu$
\quad
& 
\quad
$\Er\Du$
\quad
&
\quad
$\Er\Eu$
\quad
\\ 
\hline
\end{tabular}
\end{align*}
After sampling through the bulk vertex in the middle, we get local terms of signed measures on  outgoing configurations of $\widetilde{\hP}$ written in the first row, which should match with the local terms in the matrix ansatz of $\mu_{\widetilde{\hP}}$ in the second row:
\begin{align*}
\begin{tabular}{|c|c|c|c|}
\hline 
& & &
\\ 
\quad
\raisebox{-4pt}{\begin{tikzpicture}[scale=0.5]
		\draw[dotted] (0,0)--(2,0);\draw[dotted] (0,1)--(2,1);\draw[dotted] (0,2)--(2,2);
            \draw[dotted] (0,0)--(0,2);\draw[dotted] (1,0)--(1,2);\draw[dotted] (2,0)--(2,2);
		\draw[very thick] (0,1) -- (1,1) -- (1,0);
            \draw[path,thin] (1,1)--(2,1); \draw[path,thin] (1,1)--(1,2);
            \end{tikzpicture}}
\quad
&
\quad
\raisebox{-4pt}{\begin{tikzpicture}[scale=0.5]
		\draw[dotted] (0,0)--(2,0);\draw[dotted] (0,1)--(2,1);\draw[dotted] (0,2)--(2,2);
            \draw[dotted] (0,0)--(0,2);\draw[dotted] (1,0)--(1,2);\draw[dotted] (2,0)--(2,2);
		\draw[very thick] (0,1) -- (1,1) -- (1,0);
           \draw[path,thin] (1,1)--(1,2);
            \end{tikzpicture}}
\quad
&
\quad
\raisebox{-4pt}{\begin{tikzpicture}[scale=0.5]
		\draw[dotted] (0,0)--(2,0);\draw[dotted] (0,1)--(2,1);\draw[dotted] (0,2)--(2,2);
            \draw[dotted] (0,0)--(0,2);\draw[dotted] (1,0)--(1,2);\draw[dotted] (2,0)--(2,2);
		\draw[very thick] (0,1) -- (1,1) -- (1,0);
            \draw[path,thin] (1,1)--(2,1); 
            \end{tikzpicture}}
\quad
&
\quad
\raisebox{-4pt}{\begin{tikzpicture}[scale=0.5]
		\draw[dotted] (0,0)--(2,0);\draw[dotted] (0,1)--(2,1);\draw[dotted] (0,2)--(2,2);
            \draw[dotted] (0,0)--(0,2);\draw[dotted] (1,0)--(1,2);\draw[dotted] (2,0)--(2,2);
		\draw[very thick] (0,1) -- (1,1) -- (1,0);
            \end{tikzpicture}}
  \quad
\\[0.3cm]
\quad
$\Dr\Du$
\quad
& 
\quad
$(1-\th_2)D^\r E^\u+\th_1E^\r D^\u$
\quad
& 
\quad
$(1-\th_1)E^\r D^\u+\th_2D^\r E^\u$
\quad
&
\quad
$\Er\Eu$
\quad
\\ [0.1cm]
\quad 
$\Du\Dr$
\quad
& 
\quad
$\Du\Er$
\quad
& 
\quad
$\Eu\Dr$
\quad
&
\quad
$\Eu\Er$
\quad
\\ 
\hline
\end{tabular}
\end{align*}
This give us the bulk compatibility relations \eqref{eq:bulk relations}. 

In a left boundary local move \eqref{eq:local move left boundary}, the outgoing edges of $\hP$ and $\widetilde{\hP}$ coincide except for the leftmost edges. 
Here are the leftmost terms in the matrix ansatz of $\mu_{\hP}$ on possible local configurations:
\begin{align*}
\begin{tabular}{|c|c|}
\hline 
& 
\\ 
\quad
\raisebox{-4pt}{\begin{tikzpicture}[scale=0.6]
		\draw[dotted] (0,0) -- (1,1);\draw[dotted] (0,0) -- (1,0);
          \draw[dotted] (1,0)--(2,0)--(2,1)--(1,1)--(1,0);
            \draw[very thick] (0,0) -- (1,0);
            \draw[path,thin] (1,0)--(1,1);
		\end{tikzpicture}}
\quad
&
\quad
\raisebox{-4pt}{\begin{tikzpicture}[scale=0.6]
		\draw[dotted] (0,0) -- (1,1);\draw[dotted] (0,0) -- (1,0);
          \draw[dotted] (1,0)--(2,0)--(2,1)--(1,1)--(1,0);
            \draw[very thick] (0,0) -- (1,0);
		\end{tikzpicture}}
  \quad
\\[0.3cm]
$\ll W|\Du$
\quad
& 
\quad
$\ll W|\Eu$
\quad
\\ 
\hline
\end{tabular}
\end{align*}
After sampling through the left boundary vertex, we get the leftmost terms of signed measures on $\widetilde{\hP}$ in the first row, which should match leftmost terms in the matrix ansatz of $\mu_{\widetilde{\hP}}$ in second row:
\begin{align*}
\begin{tabular}{|c|c|}
\hline 
& 
\\
\quad
\raisebox{-4pt}{\begin{tikzpicture}[scale=0.6]
		\draw[dotted] (0,0) -- (1,1);\draw[dotted] (0,0) -- (1,0);
          \draw[dotted] (1,0)--(2,0)--(2,1)--(1,1)--(1,0);
            \draw[very thick] (1,0) -- (1,1);
            \draw[path,thin] (1,1)--(2,1);
		\end{tikzpicture}}
\quad
&
\quad
\raisebox{-4pt}{\begin{tikzpicture}[scale=0.6]
		\draw[dotted] (0,0) -- (1,1);\draw[dotted] (0,0) -- (1,0);
          \draw[dotted] (1,0)--(2,0)--(2,1)--(1,1)--(1,0);
            \draw[very thick] (1,0) -- (1,1);
		\end{tikzpicture}}
\quad
\\[0.3cm]
\quad
$(1-c)\ll W|\Du+a\ll W|\Eu$
\quad
& 
\quad
$(1-a)\ll W|\Eu+c\ll W|\Du$
\quad
\\[0.1cm]
\quad
$\ll W|\Dr$
\quad
& 
\quad
$\ll W|\Er$
\quad\\
\hline
\end{tabular}
\end{align*}
This gives us the left boundary compatibility relations \eqref{eq:left boundary relations}. 
The right boundary compatibility relations \eqref{eq:right boundary relations} can be obtained similarly.
\end{proof}

The compatibility relations \eqref{eq:bulk relations}, \eqref{eq:left boundary relations} and \eqref{eq:right boundary relations} 
look complicated, but in fact they can be reduced to the $\dehp$ algebra \eqref{eq:DEHP algebra} after imposing some simple (additional) relations \eqref{eq:additional assumptions}. 
\begin{theorem}\label{thm:stationary measure S6V}
Assume  $b+d<1$.
 Suppose $\D$ and $\E$ are matrices, $\ll W|$ is a row vector and $|V\rr$ is a column vector with the same (possibly infinite) dimension, satisfying the $\dehp$ algebra \eqref{eq:DEHP algebra}:
    \begin{equation*} 
        \D\E-q\E\D=\D+\E, \quad
        \ll W|(\alpha\E-\gamma\D)=\ll W|, \quad
        (\beta\D-\de\E)|V\rr=|V\rr,\end{equation*} 
    with parameters:
    \be \label{eq: abcd to alphabetathetatheta}
    (q,\alpha,\beta,\gamma,\de)=\lb\frac{\th_1}{\th_2},\frac{(1-\th_1)a}{\th_2},\frac{(1-\th_2)b}{\th_2(1-b-d)},\frac{(1-\th_2)c}{\th_2},\frac{(1-\th_1)d}{\th_2(1-b-d)}\rb.
    \ee
    Then the matrices
    \be \label{eq:substitution}
D^\u=\frac{(1-\th_2)}{\th_2}\D,\quad E^\u=\frac{(1-\th_1)}{\th_2}\E,\quad D^\r=D^\u+I,\quad E^\r=E^\u-I,
\ee
together with boundary vectors $\ll W|$, $|V\rr$ satisfy the compatibility relations \eqref{eq:bulk relations}, \eqref{eq:left boundary relations} and \eqref{eq:right boundary relations}. 
\end{theorem}
\begin{proof}
We take  
\be\label{eq:additional assumptions} D^\r=D^\u+I,\quad E^\r=E^\u-I\ee
into compatibility relations \eqref{eq:bulk relations}, \eqref{eq:left boundary relations} and \eqref{eq:right boundary relations}. They simplify to  three equations:
\begin{equation}\label{eq:in the proof similar to DEHP}
    \begin{gathered}
        \th_2D^\u E^\u-\th_1E^\u D^\u=(1-\th_2)E^\u+(1-\th_1)D^\u,\\
        \ll W|=\ll W|(aE^\u-cD^\u),\\
        (1-b-d)|V\rr=(bD^\u-dE^\u)|V\rr,
    \end{gathered}
\end{equation}
which look similar to the $\dehp$ algebra \eqref{eq:DEHP algebra}. In fact, if we set the parameters $(q,\alpha,\beta,\gamma,\de)$ as \eqref{eq: abcd to alphabetathetatheta} and make the substitution
$$D^\u=\frac{(1-\th_2)}{\th_2}\D,\quad E^\u=\frac{(1-\th_1)}{\th_2}\E,$$
then \eqref{eq:in the proof similar to DEHP} become exactly the $\dehp$ algebra \eqref{eq:DEHP algebra}. Hence if we start with a solution $\D,\E,\ll W|,|V\rr$ of the $\dehp$ algebra with parameters $(q,\alpha,\beta,\gamma,\delta)$ as \eqref{eq: abcd to alphabetathetatheta} then the matrices $D^\u,D^\r,E^\u,E^\r$ as \eqref{eq:substitution} and the same boundary vectors $\ll W|$, $|V\rr$ satisfy compatibility relations \eqref{eq:bulk relations}, \eqref{eq:left boundary relations} and \eqref{eq:right boundary relations}.
\end{proof}


 We now provide the proof of Theorem \ref{thm:main thm tilting} in the introduction. 
\begin{proof}[Proof of Theorem \ref{thm:main thm tilting}]
Recall that by Theorem \ref{thm:stationary measure S6V}, a solution of the compatibility relations in Theorem \ref{thm: general matrix ansatz construction} can be given by a solution of the $\dehp$ algebra. 
 Conditions \eqref{eq:condition Thm 1.1} guarantee that $(q,\alpha,\beta,\gamma,\delta)$ given by \eqref{eq: abcd to alphabetathetatheta} above satisfy the usual constraints \eqref{eq:conditions open ASEP} of open ASEP. 
When $\hP$ is a horizontal path we have $p_1=\dots=p_N=\uparrow$. Since
      $$D^\u=\frac{(1-\th_2)}{\th_2}\D,\quad E^\u=\frac{(1-\th_1)}{\th_2}\E.$$
      The stationary measure $\mu$ can be written as
\be \label{eq:in proof of main thm}
\begin{split}
    \mu(\tau_1,\dots,\tau_N)&=\frac{\ll W|((1-\tau_1)\Eu+\tau_1\Du)\times\dots\times ((1-\tau_N)\Eu+\tau_N\Du)|V\rr}{\ll W|\prod_{i=1}^N(\Eu+\Du)|V\rr}\\
    &=r^{\sum_{i=1}^N\tau_i}\frac{\ll W|((1-\tau_1)\E+\tau_1\D)\times\dots\times ((1-\tau_1)\E+\tau_1\D)|V\rr}{\ll W|\prod_{i=1}^N(\E+r\D)|V\rr}\\
    &=r^{\sum_{i=1}^N\tau_i}\pi(\tau_1,\dots,\tau_N)/Z,
\end{split}
\ee
for $r=\frac{1-\th_2}{1-\th_1}$ and normalizing constant $Z$ given by:
 
$$Z=\sum_{(\tau_1,\dots,\tau_N)\in\left\{0,1\right\}^N}r^{\sum_{i=1}^N\tau_i}\pi(\tau_1,\dots,\tau_N).$$
 
We used Theorem \ref{thm:matrix ansatz from DEHP} in the last step  of \eqref{eq:in proof of main thm}.
The assumption $ab/(cd)\notin\{q^l:l=0,1,\dots\}$ is equivalent to the `non-singular' condition \eqref{eq:singular case} of $(q,\alpha,\beta,\gamma,\delta)$, which, by Remark \ref{remark:nonsingular}, guarantees that the denominator of \eqref{eq:in proof of main thm} above is nonzero.
 
\end{proof}
\begin{remark}
    In Theorem \ref{thm:main thm tilting}  
    and in section \ref{sec 3}  we only consider the case when $\th_1<\th_2$.
    When $\th_1>\th_2$ we can still get the stationary measure by a standard particle-hole duality argument. More precisely, when we swap the parameters in the six-vertex model:
    $$ 
    \th_1\longleftrightarrow\th_2,\quad a\longleftrightarrow c,\quad b\longleftrightarrow d,
    $$
    any edge equipped with $\tau\in\left\{0,1\right\}$ arrow becomes equipped with $1-\tau$ arrow. Hence in the   particle systems particles become holes and holes become particles. The stationary measures are related by
     $\mu(\tau_1,\dots,\tau_N)=\nu(1-\tau_1,\dots,1-\tau_N),$  
    for any $(\tau_1,\dots,\tau_N)\in\left\{0,1\right\}^N$.
\end{remark}
\subsection{Bernoulli and \texorpdfstring{$q$}{}-volume stationary measures in special cases}
\label{subsec:stationary measure special cases}
Based on the general matrix product solution of the stationary measure on a down-right path in Theorem \ref{thm: general matrix ansatz construction}, we obtain the Bernoulli and the $q$-volume stationary measures in some special cases. In these cases we do not necessarily have $a,b,c,d\in(0,1)$ so the stationary measure may not be unique.
\begin{corollary}
    Suppose that
\be \label{eq:condition bernoulli stationary measure}
(1-\th_1)(a+d-ab-ad)(b+c-bc-ab)=(1-\th_2)(a+d-ad-cd)(b+c-bc-cd),\ee
and that both sides of the equation are non-zero.
Suppose $\hP$ is a down-right path on the strip with outgoing edges $p_1,\dots,p_N\in\left\{\u,\r\right\}$.
Then we have a stationary measure $\mu_{\hP}$ of six-vertex model on a strip on down-right path $\hP$, which is Bernoulli with probability 
$$p^\u=\frac{a+d-ab-ad}{a+b+c+d-(a+c)(b+d)},$$
on the sites $\tau_i$ where $p_i=\u$, and Bernoulli with probability
$$p^\r=\frac{a+d-ad-cd}{a+b+c+d-(a+c)(b+d)},$$
on the sites $\tau_i$ where $p_i=\r$.
\end{corollary}
\begin{proof}
    One can observe that in the case of \eqref{eq:condition bernoulli stationary measure}, the $1$-dimensional matrices
    $$D^\u=a+d-ab-ad,\quad E^\u=b+c-bc-cd,\quad D^\r=a+d-ad-cd,\quad E^\r=b+c-bc-ab,$$
    and vectors $\ll W|=|V\rr=1$ 
    satisfy compatibility relations \eqref{eq:bulk relations}, \eqref{eq:left boundary relations} and \eqref{eq:right boundary relations}. Hence we get the Bernoulli stationary measure from Theorem \ref{thm: general matrix ansatz construction}.
\end{proof}
\begin{corollary}
    When $a=b=c=d=1$ our model has `anti-reflecting' boundary, i.e. an arrow that touches the boundary must exist the system, and an arrow enters the system at a boundary point exactly when no arrow exists at this boundary point. In the corresponding interacting particle system on a down-right path, the parity of the number of particles get preserved. Suppose $(1-\th_1)(1-\th_2)\neq 0$ and $\hP$ is a down-right path with outgoing edges $p_1,\dots,p_N\in\left\{\u,\r\right\}$.
Then we have a stationary measure $\mu_{\hP}$ of the six-vertex model on the strip on $\hP$, which is Bernoulli with probability 
$$p^\u=\frac{\sqrt{1-\th_2}}{\sqrt{1-\th_1}+\sqrt{1-\th_2}}$$
on the sites $\tau_i$ where $p_i=\u$, and Bernoulli with probability
$$p^\r=\frac{\sqrt{1-\th_1}}{\sqrt{1-\th_1}+\sqrt{1-\th_2}}$$
on the sites $\tau_i$ where $p_i=\r$.

    We can get two stationary measures from it by restricting this measure to the set of states with even/odd number of particles and multiplying a normalizing constant.
\end{corollary}
\begin{proof}
    When $a=b=c=d=1$ one can observe that the $1$-dimensional matrices
    $$D^\r=E^\u=\sqrt{1-\th_1},\quad D^\u=E^\r=\sqrt{1-\th_2},$$
    and vectors $\ll W|=|V\rr=1$ 
    satisfy compatibility relations \eqref{eq:bulk relations}, \eqref{eq:left boundary relations} and \eqref{eq:right boundary relations}. Hence we get the Bernoulli stationary measure from Theorem \ref{thm: general matrix ansatz construction}.
\end{proof}
\begin{corollary}
    When $a=b=c=d=0$ our model has `reflecting' boundary, i.e. any arrow that touches the boundary must bounce back, and no new arrows can be created at the boundary. In the interacting particle system on a down-right path, the total number of particles is preserved. Let $q=\th_1/\th_2$ and suppose $0\leq k\leq N$ is the number of particles in the system. Then we have the   collection of stationary measures $\mu_k$ for $0\leq k\leq N$ on any down-right path $\hP$, with probability
      $$\frac{q^{-\sum m_j}}{\sum_{1\leq \ell_1<\dots<\ell_k\leq N}q^{-\sum \ell_j}}$$
    at the state where the $k$ particles are placed at the sites $1\leq m_1<\dots<m_k\leq N$.
\end{corollary}
\begin{proof}
    When $a=b=c=d=0$ we set $$D^\r=D^\u=D,\quad E^\r=E^\u=E.$$ The compatibility relations \eqref{eq:bulk relations}, \eqref{eq:left boundary relations} and \eqref{eq:right boundary relations} turn into a single relation
    $$DE=qED.$$
    We take the representation on the Fock space spanned by $\left\{e_i:i\in\mathbb{Z}_{\geq 0}\right\}$:
    $$E=\sum_{n=0}^\infty q^n|e_n\rr\ll e_n|,\quad D=\sum_{n=1}^\infty|e_{n-1}\rr\ll e_n|,$$
    and $W=e_0$, $V=e_k$. Theorem \ref{thm: general matrix ansatz construction} gives the stationary measure $\mu_k$.
\end{proof}
\section{Askey-Wilson processes and limit of the mean particle density}\label{sec 3}
After the matrix product ansatz solution of stationary measure of open ASEP in \cite{DEHP93}  as reviewed in subsection \ref{subsec:matrix ansatz for open ASEP}), there are various representations \cite{DEHP93,USW04,ED,Sandow,Sasamoto,BECE,ER} of the $\dehp$ algebra \eqref{eq:DEHP algebra} for different parameters $(q,\alpha,\beta,\gamma,\delta)$, which induce studies of asymptotics of open ASEP as number of sites $N\rightarrow\infty$.  As mentioned in Remark \ref{remark:USW}, the USW  representation for general parameters $(q,\alpha,\beta,\gamma,\delta)$ was given in the seminal work \cite{USW04} in terms of the  Askey-Wilson orthogonal polynomials. Using the USW representation, the open ASEP stationary measure was written in \cite{BW17} as expectations of the Askey-Wilson Markov process introduced in \cite{BW10} (Theorem \eqref{thm:askey wilson for open asep}).  
Many asymptotics of open ASEP was then rigorously studied in \cite{BW17,BW19} by this technique. We briefly review Askey-Wilson processes and the phase diagram of stationary measure of open ASEP in first two subsections, following \cite{BW10,BW17,BW19}. In the third subsection we prove Theorem \ref{thm:main theorem section 3} on the mean density of stationary measure of six-vertex model on a strip on a horizontal path, which in particular provides the phase diagram (Figure \ref{fig:phase diagram}). 



\subsection{Backgrounds on Askey-Wilson process}
\label{subsec:background AW process}
The Askey-Wilson measures are probability measures which make the Askey–Wilson polynomials orthogonal. Based on these measures, \cite{BW10} introduced a family of time-\emph{inhomogeneous}  Markov processes called Askey-Wilson processes.

The Askey-Wilson measures depend on five parameters $(\aa,\bb,\cc,\dd,q)$, where $q\in(-1,1)$ and the parameters $\aa,\bb,\cc,\dd$ admit the following three possibilities:
\begin{enumerate}
    \item [(1)] all of them are real,
    \item [(2)] two of them are real and the other two form a complex conjugate pair,
    \item [(3)] they form two complex conjugate pairs,
\end{enumerate}
and in addition we require:
$$\aa\cc, \aa\dd, \bb\cc, \bb\dd, q\aa\cc, q\aa\dd, q\bb\cc, q\bb\dd, \aa\bb\cc\dd, q\aa\bb\cc\dd\in\mathbb{C}\setminus[1,\infty).$$
The Askey-Wilson measure is of mixed type:
$$\nu(dy;\aa, \bb, \cc, \dd, q)=f(y, \aa, \bb, \cc, \dd, q)dy+\sum_{z\in F(\aa,\bb,\cc,\dd,q)}p(z)\delta_z(dy),$$
with absolutely continuous part supported on $[-1,1]$ with density 
\begin{equation}\label{eq:defn of continuous part density}
    f(y,\aa,\bb,\cc,\dd,q)=\frac{(q,\aa\bb,\aa\cc,\aa\dd,\bb\cc,\bb\dd,\cc\dd;q)_{\infty}}{2\pi(\aa\bb\cc\dd;q)_{\infty}\sqrt{1-y^2}}\left\vert\frac{(e^{2i\th_y};q)_{\infty}}{(\aa e^{i\th_y},\bb e^{i\th_y},\cc e^{i\th_y},\dd e^{i\th_y};q)_{\infty}}\right\vert^2,
\end{equation}
where $y=\cos\th_y$ and  $f(y,\aa,\bb,\cc,\dd,q)=0$ when $|y|>1$.  We use the $q$-Pochhammer symbol: for complex $z,z_1,\dots,z_k$ and $0\leq n\leq\infty$, $$(z;q)_n=\prod_{j=0}^{n-1}\,(1-z q^j), \quad (z_1,\cdots,z_k;q)_n=\prod_{i=1}^k(z_i;q)_n.$$ 
The discrete (atomic) part supported on a finite or empty set $F(\aa,\bb,\cc,\dd,q)$ of atoms generated by numbers $\chi\in\left\{\aa,\bb,\cc,\dd\right\}$ such that $|\chi|>1$. 
In this case $\chi$ must be real and generates its own set of atoms:
\begin{equation}\label{eq:defn of atoms}
    y_j=\frac{1}{2}\left(\chi q^j+\frac{1}{\chi q^j}\right) \text{ for }j=0,1\dots\text{ such that } |\chi q^j|> 1,
\end{equation}
the union of which is $F(\aa,\bb,\cc,\dd,q)$. 
When $\chi=\aa$, the corresponding masses are
$$
p(y_0;\aa,\bb,\cc,\dd,q) =\frac{\qp{\aa^{-2},\bb\cc,\bb\dd,\cc\dd}}{\qp{\bb/\aa,\cc/\aa,\dd/\aa,\aa\bb\cc\dd}},
$$
$$
p(y_j;\aa,\bb,\cc,\dd,q) =p(y_0;\aa,\bb,\cc,\dd,q)\frac{\qps{\aa^2,\aa\bb,\aa\cc,\aa\dd}_j\,(1-\aa^2q^{2j})}{\qps{q,q\aa/\bb,q\aa/\cc,q\aa/\dd}_j(1-\aa^2)}\left(\frac{q}{\aa\bb\cc\dd}\right)^j,\quad j\geq 1.
$$
For other values of $\chi$ the masses are given by similar formulas with $\aa$ and $\chi$ swapped.
We will not use these precise formulas of the masses in this paper.

The Askey-Wilson processes 
depend on five parameters $(\AA,\BB,\CC,\DD,q)$, where $q\in(-1,1)$ and $\AA,\BB,\CC,\DD$ are either real or $(\AA,\BB)$ or $(\CC,\DD)$ are complex conjugate pairs, and in addition
$$\AA\CC,\AA\DD,\BB\CC,\BB\DD,q\AA\CC,q\AA\DD,q\BB\CC,q\BB\DD,\AA\BB\CC\DD,q\AA\BB\CC\DD\in\mathbb{C}\setminus[1,\infty).$$
The Askey-Wilson process $\left\{Y_t\right\}_{t\in I}$ is a  time-\emph{inhomogeneous} Markov process defined on the interval
$$
  I=\left[\max\{0,\CC\DD,q\CC\DD\},\,\tfrac{1}{\max\{0,\AA\BB,q\AA\BB\}}\right),
$$ with marginal distributions
$$
  \pi_t(dx)=\nu(dx;\AA\sqrt{t},\BB\sqrt{t},\CC/\sqrt{t},\DD/\sqrt{t},q),\quad t\in I,
$$ (which has compact support $U_t$) and transition probabilities
$$
   P_{s,t}(x,dy)=\nu(dy;\AA\sqrt{t},\BB\sqrt{t},\sqrt{s/t}(x+\sqrt{x^2-1}),\sqrt{s/t}(x-\sqrt{x^2-1})),
$$
for $s<t$, $s,t\in I$, $x\in U_s$. We remark that the marginal distribution $\pi_t(dx)$ may have atoms at 
\begin{equation}\label{eq:defn of all possible atoms}
    \frac{1}{2}\left(\AA\sqrt{t} q^j+\frac{1}{\AA\sqrt{t} q^j}\right),\quad \frac{1}{2}\left(\BB\sqrt{t} q^j+\frac{1}{\BB\sqrt{t} q^j}\right),\quad\frac{1}{2}\left(\frac{\CC q^j}{\sqrt{t}}+\frac{\sqrt{t}}{\CC q^j}\right),\quad \frac{1}{2}\left(\frac{\DD q^j}{\sqrt{t}}+\frac{\sqrt{t}}{\DD q^j}\right),
\end{equation}
and the transition probabilities $P_{s,t}(x,dy)$ may also have atoms.

In the following subsections we only consider Askey-Wilson process under conditions
$$\AA,\CC\geq 0,\quad\BB,\DD\leq 0,\quad \AA\CC<1,\quad\BB\DD<1,$$
in which case the process is defined on interval $I=[0,\infty)$, and the marginal distributions $\pi_t(dx)$ cannot be purely discrete.
\subsection{Stationary measure of open ASEP and asymptotics}
\label{subsec:background asymptotics open ASEP}
We consider open ASEP on the lattice $\left\{1,\dots,N\right\}$ with particle jump rates $(q,\alpha,\beta,\gamma,\th)$ satisfying \eqref{eq:conditions open ASEP}. 
\begin{definition}\label{defn:parameterization open ASEP}
    We will use the following parameterization:
\be\label{eq:defining ABCD}
A=\k_+(\beta,\de),\quad B=\k_-(\beta,\de),\quad C=\k_+(\alpha,\gamma),\quad D=\k_-(\alpha,\gamma),
\ee
where 
$$
\k_{\pm}(u,v)=\frac{1}{2u}\lb1-q-u+v\pm\sqrt{(1-q-u+v)^2+4uv}\rb.
$$
We can check that for any given $q\in[0,1)$, \eqref{eq:defining ABCD} gives a bijection 
$$\left\{(\alpha,\beta,\gamma,\delta):\alpha,\beta>0, \gamma,\delta\geq 0\right\}\overset{\sim}{\longrightarrow}\left\{(A,B,C,D): A,C\geq0, B,D\in(-1,0]\right\}.$$
\end{definition}

\begin{theorem}[Theorem 1 in \cite{BW17}]\label{thm:askey wilson for open asep}
    Suppose $AC<1$ and $\D,\E,\ll W|,|V\rr$ satisfy the $\dehp$ algebra with $\ll W|V\rr=1$. 
    Then for $0< t_1\leq\dots\leq t_N$, we have
    \be\label{eq:MP}\ll W|\prod_{j=1}^{\substack{N\\\longrightarrow}}(\E+t_j\D)|V\rr=\frac{1}{(1-q)^N}\mathbb{E}\lb\prod_{j=1}^N(1+t_j+2\sqrt{t_j}Y_{t_j})\rb,\ee
    where the right arrow means that the product is taken in increasing order of $j$ from left to right.
    Hence by Theorem \ref{thm:matrix ansatz from DEHP}, the generating function of stationary measure of open ASEP reads:
    \be\label{eq:generating function pi}
    \mathbb{E}_{\pi}\lb\prod_{j=1}^Nt_j^{\tau_j}\rb=\frac{\mathbb{E}\lb\prod_{j=1}^N(1+t_j+2\sqrt{t_j}Y_{t_j})\rb}{2^N\mathbb{E}(1+Y_1)^N},
    \ee 
    where $\left\{Y_t\right\}_{t\geq 0}$ is the Askey–Wilson process with parameters $(A, B, C, D, q)$.
\end{theorem}

\begin{remark}\label{remark:fan region implies nonsingular}
We note that the above theorem was proved in \cite{BW17} for $\D$, $\E$, $\ll W|$ and $|V\rr$ given by the USW representation \cite{USW04} in Remark \ref{remark:USW}. In view of $AC<1$ we have $\gamma\delta/(\alpha\beta)=ABCD<1$, in particular $\alpha\beta\neq q^l\gamma\delta$ for any $l=0,1,\dots$. By Remark \ref{remark:nonsingular} we have that the matrix product \eqref{eq:MP} does not depend on the specific choice of $\D$, $\E$, $\ll W|$ and $|V\rr$ satisfying the DEHP algebra, and that one only need to assume $\ll W|V\rr=1$.
\end{remark}

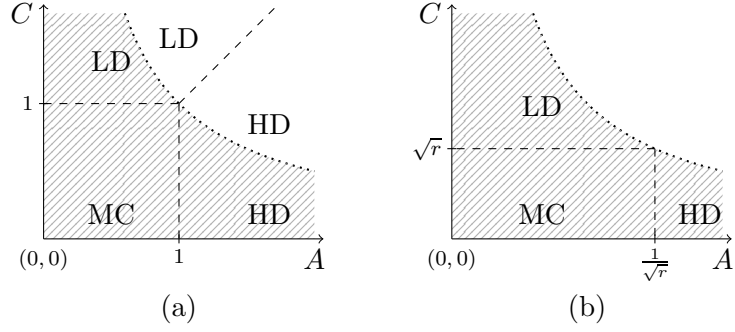
\begin{figure}
\begin{center}
\begin{tikzpicture}[scale=0.6]
\draw[scale = 1,domain=6.8:11,smooth,variable=\x,dotted,thick] plot ({\x},{1/((\x-7)*1/3+2/3)*3+5});
\fill[pattern=north east lines, pattern color=gray!60] (5,5)--(5,10) -- plot [domain=6.8:11]  ({\x},{1/((\x-7)*1/3+2/3)*3+5}) -- (11,5) -- cycle;
 \draw[->] (5,5) to (5,10.2);
 \draw[->] (5.,5) to (11.2,5);
   \draw[-, dashed] (5,8) to (8,8);
   \draw[-, dashed] (8,8) to (8,5);
   \draw[-, dashed] (8,8) to (10.2,10.2);
   \node [left] at (5,8) {\scriptsize$1$};
   \node[below] at (8,5) {\scriptsize $1$};
     \node [below] at (11,5) {$A$};
   \node [left] at (5,10) {$C$};
  \draw[-] (8,4.9) to (8,5.1);
   \draw[-] (4.9,8) to (5.1,8);
 \node [below] at (5,5) {\scriptsize$(0,0)$};
    \node [above] at (6.5,8.5) {LD};\node [above] at (8,9) {LD};
    \node [below] at (10,6) {HD};  \node [below] at (10,8) {HD}; 
 \node [below] at (6.5,6) {MC};
  \node[below] at (8,4){(a)};
\end{tikzpicture}
\quad\quad
\begin{tikzpicture}[scale=0.6]
\draw[scale = 1,domain=6.8:11,smooth,variable=\x,dotted,thick] plot ({\x},{1/((\x-7)*1/3+2/3)*3+5});
\fill[pattern=north east lines, pattern color=gray!60] (5,5)--(5,10) -- plot [domain=6.8:11]  ({\x},{1/((\x-7)*1/3+2/3)*3+5}) -- (11,5) -- cycle;
 \draw[->] (5,5) to (5,10.2);
 \draw[->] (5.,5) to (11.2,5);
   \draw[-, dashed] (9.5,7) to (5,7);
   \draw[-, dashed] (9.5,7) to (9.5,5);
   \node [left] at (5,7) {\scriptsize$\sqrt{r}$};
   \node[below] at (9.5,5) {\scriptsize $\frac{1}{\sqrt{r}}$};
     \node [below] at (11,5) {$A$};
   \node [left] at (5,10) {$C$};
  \draw[-] (9.5,4.9) to (9.5,5.1);
   \draw[-] (4.9,7) to (5.1,7);
 \node [below] at (5,5) {\scriptsize$(0,0)$};
    \node [above] at (7,7.5) {LD};
    \node [below] at (10.5,6) {HD}; 
 \node [below] at (7,6) {MC};
 \node[below] at (8,4){(b)};
\end{tikzpicture}
\end{center}
\caption{(a) Phase diagram for open ASEP. (b) Phase diagram for six-vertex model on a strip.} 
\label{fig:phase diagram}
\end{figure}
It has been well-known in the physics literature (see for example \cite{DDM, Sandow,DLS1,DLS2}) that the limit of the statistics of open ASEP under the stationary measure as the number of sites $N\rightarrow\infty$ exhibit a phase diagram that involves only two boundary parameters $A,C$:
\begin{definition}
We define the phase diagram of open ASEP involving parameters $A,C$.

    We define two regions:
    \begin{enumerate}
        \item [$\bullet$] (fan region) $AC<1$,
        \item [$\bullet$] (shock region) $AC>1$.
    \end{enumerate}
    
    We also define three phases:
    \begin{enumerate}
        \item [$\bullet$] (maximal current phase) $A<1$, $C<1$,
        \item [$\bullet$] (high density phase) $A>1$, $A>C$,
        \item [$\bullet$] (low density phase) $C>1$, $C>A$.
    \end{enumerate}
    See Figure \ref{fig:phase diagram} (a) where the shadowed area denote the fan region, and the three phases are labeled.
\end{definition}

\begin{definition}\label{defn:mean particle density}
    Consider a particle system with $N$ sites with occupation variables $(\tau_1,\dots,\tau_N)\in\left\{0,1\right\}^N$. 
    The observable $\rho=\frac{1}{N}\sum_{i=1}^N\tau_i$ is called the mean particle density of the system.
\end{definition}
The following limit of the mean particle density has been well-known in physics, see for example \cite{DLS1,SD}. It was obtained in \cite{USW04} by Askey-Wilson polynomials and later in mathematical works \cite{BW17,BW19} by Askey-Wilson processes. More specific asymptotics are also obtained therein.
\begin{theorem} \label{thm: open ASEP particle density}
    Consider the open ASEP with sites $\left\{1,\dots,N\right\}$ and with particle jump rates $(q,\alpha,\beta,\gamma,\delta)$ whose stationary measure is denoted by $\pi(\tau_1,\dots,\tau_N)$. On the fan region $AC<1$ the limits of mean particle density as $N\rightarrow\infty$ are given by:
    \[
\lim_{N\rightarrow\infty}\mathbb{E}_\pi\rho=
\begin{cases}
\frac{1}{2}, & A\leq 1, C\leq 1 \text{ (maximal current phase with boundary)},\\
\frac{A}{A+1}, & A>1, A>C \text{ (high density phase)},\\
\frac{1}{C+1}, & C>1, C>A \text{ (low density phase)}.
\end{cases}
\]
\end{theorem}
\subsection{Limit of mean particle density of six-vertex model on a strip} 
\label{subsec:proof of limit of mean density}
 In this subsection we prove Theorem \ref{thm:main theorem section 3} in the introduction, which gives the limits of mean particle density of the stationary measure of six-vertex model on a strip. The limits exhibit a phase diagram that is a tilting of phase diagram of open ASEP (Figure \ref{fig:phase diagram}).
 

We first fix some notations. We will consider the six-vertex model on a strip with bulk parameters $\th_1, \th_2$ and boundary parameters $a, b, c, d$. We will always assume:
    $$a, b, c, d, \th_1, \th_2\in(0,1),\quad \th_1<\th_2,\quad b+d<1.$$
    We define parameter $r=\frac{1-\th_2}{1-\th_1}\in(0,1)$ 
    and parameters $q, \alpha,\beta,\gamma,\delta$ by \eqref{eq: abcd to alphabetathetatheta}:
    $$(q,\alpha,\beta,\gamma,\de)=\lb\frac{\th_1}{\th_2},\frac{(1-\th_1)a}{\th_2},\frac{(1-\th_2)b}{\th_2(1-b-d)},\frac{(1-\th_2)c}{\th_2},\frac{(1-\th_1)d}{\th_2(1-b-d)}\rb.$$
    Last we define $A,B,C,D$ in terms of $(q,\alpha,\beta,\gamma,\de)$ by Definition \ref{defn:parameterization open ASEP}.
We denote the stationary measure on a horizontal down-right path by $\mu(\tau_1,\dots,\tau_N)$.

 We assume that the matrices $\D,\E$, row vector $\ll W|$ and column vector $|V\rr$ satisfy the $\dehp$ algebra \eqref{eq:DEHP algebra} with parameters $(q,\alpha,\beta,\gamma,\delta)$, and that $\ll W|V\rr=1$. As mentioned in Remark \ref{remark:USW}, a concrete example is given by the $\usw$ representation \cite{USW04}. 
Then  by  Theorem \ref{thm:stationary measure S6V} we have that the matrices
\be\label{eq:relations of DE with DuEuDrEr in section 3}
D^\u=\frac{(1-\th_2)}{\th_2}\D,\quad E^\u=\frac{(1-\th_1)}{\th_2}\E,\quad D^\r=D^\u+I,\quad E^\r=E^\u-I,
\ee
and the same boundary vectors $\ll W|,|V\rr$ satisfy compatibility relations \eqref{eq:bulk relations}, \eqref{eq:left boundary relations} and \eqref{eq:right boundary relations} in Theorem \ref{thm: general matrix ansatz construction}. 
\begin{corollary}\label{cor:generating function stationary measure S6V}
    Suppose $AC<1$. Then for any $0< t_1\leq\dots\leq t_N$, 
    $$\ll W|\prod_{j=1}^{\substack{N\\\longrightarrow}}(\Du t_j+\Eu)|V\rr=\lb\frac{1-\th_1}{\th_2(1-q)}\rb^N\mathbb{E}\lb\prod_{j=1}^N(1+rt_j+2\sqrt{rt_j}Y_{rt_j})\rb.$$
    Hence the generating function of stationary measure $\mu$ can be written as
    \be\label{eq:generating function mu}
    \mathbb{E}_{\mu}\lb\prod_{j=1}^Nt_j^{\tau_j}\rb=\frac{\mathbb{E}\lb\prod_{j=1}^N(1+rt_j+2\sqrt{rt_j}Y_{rt_j})\rb}{\mathbb{E}(1+r+2\sqrt{r}Y_{r})^N},
    \ee 
    where $\left\{Y_t\right\}_{t\geq 0}$ is the Askey–Wilson process with parameters $(A, B, C, D, q)$.
\end{corollary}
\begin{proof}
    This is an easy consequence of equation \eqref{eq:relations of DE with DuEuDrEr in section 3}, Theorem \ref{thm:askey wilson for open asep} and Theorem \ref{thm: general matrix ansatz construction}.
     As in Remark \ref{remark:fan region implies nonsingular}, $AC<1$ implies $\alpha\beta\neq q^l\gamma\delta$ for any $l=0,1,\dots$, which, by Remark \ref{remark:nonsingular}, guarantees that the denominator in the matrix product ansatz is nonzero. 
\end{proof}
\begin{corollary}\label{cor:particle density in terms of askey wilson}
    Consider the mean density $\rho=\frac{1}{N}\sum_{i=1}^N\tau_i$ under stationary measure $\mu$ of the  particle system of six-vertex model on a strip on a horizontal path.  If $AC<1$, then we have 
    $$\mathbb{E}_{\mu}\rho=\frac{\pt Z_N(t)\po}{NZ_N(1)},$$
    where $Z_N(t)$ is given by 
    \be \label{eq:normalizing constant in terms of Askey-Wilson}
     Z_N(t)=\lb\frac{1-\th_1}{\th_2(1-q)}\rb^N\mathbb{E}(1+rt+2\sqrt{rt}Y_{rt})^N,
    \ee
    where $\left\{Y_t\right\}_{t\geq 0}$ is the Askey–Wilson process with parameters $(A, B, C, D, q)$. 
\end{corollary}
\begin{proof}
By Theorem \ref{thm: general matrix ansatz construction}, we have: 
    $$\mathbb{E}_\mu\tau_i=\mathbb{P}_\mu(\tau_i=1)=\frac{\ll W|(\Du+\Eu)^{i-1}\Du(\Du+\Eu)^{N-i}|V\rr}{\ll W|(\Du +\Eu)^N|V\rr}.$$
    Summing over $i=1,\dots,N$ we get $$\mathbb{E}_{\mu}\rho=\frac{\pt Z_N(t)\po}{NZ_N(1)},$$
    where  
    $$Z_N(t)=\ll W|(\Du t+\Eu)^N|V\rr,\quad\forall t\geq 0.$$
    is the normalizing constant.
   By Corollary \ref{cor:generating function stationary measure S6V} we can write  $Z_N(t)$ in the form \eqref{eq:normalizing constant in terms of Askey-Wilson}.
\end{proof}

\begin{definition}
    We define the phase diagram of six-vertex model on a strip on a horizontal path.
    
    We first define two regions:
    \begin{enumerate}
        \item [$\bullet$] (fan region) $AC<1$,
        \item [$\bullet$] (shock region) $AC>1$.
    \end{enumerate}
    
    On the fan region we define three phases:
    \begin{enumerate}
        \item [$\bullet$] (maximal current phase) $A<1/\sqrt{r}$, $C<\sqrt{r}$,
        \item [$\bullet$] (high density phase) $A>1/\sqrt{r}$, $AC<1$,
        \item [$\bullet$] (low density phase) $C>\sqrt{r}$, $AC<1$.
    \end{enumerate}
    See Figure \ref{fig:phase diagram} (b) where the shadowed area denote the fan region, and the three phases are labeled.
\end{definition}
\begin{remark}\label{remark:fan region and extension}
    We have defined the three phases only inside the fan region $AC<1$, because Theorem \ref{thm:askey wilson for open asep} only works inside this region so our techniques can only produce results therein. After the initial submission of this paper, \cite{WWY} constructed  Askey-Wilson signed measures and rigorously obtained the density and fluctuations of open ASEP in the shock region $AC>1$, including the `coexistence line' $A=C>1$. It is possible that, by applying similar techniques, we can also extend the phase diagram of the six-vertex model on a strip to include the shock region. We defer this aspect to future research. 
\end{remark}

 We now provide the proof of Theorem \ref{thm:main theorem section 3} in the introduction. 
\begin{proof}[Proof of Theorem \ref{thm:main theorem section 3}]
We first prove this theorem under condition \eqref{eq:technical condition}. At the end of the proof we show that when $a+c<1$ we can avoid this assumption by a two-species attractive coupling argument (Proposition \ref{prop:coupling}) similar to the one in \cite[section 5]{Corwin_Knizel}, which is of independent interest. We denote:
    $$\tA=A\sqrt{r},\quad\tB=B\sqrt{r},\quad\tC=C/\sqrt{r},\quad\tD=D/\sqrt{r}.$$
    By  \eqref{eq:condition Thm 1.1} and $AC<1$ we have:
    \be\label{eq:conditions on tAtBtCtD}
    \tA,\tC> 0,\quad\tA\tC<1,\quad \tB\in(-\sqrt{r},0),\quad \tD\in(-1/\sqrt{r},0).
    \ee
Denote $\zt=\mathbb{E}(1+rt+2\sqrt{rt}Y_{rt})^N$. By Corollary \ref{cor:particle density in terms of askey wilson} we have $$Z_N(t)=\lb\frac{1-\th_1}{\th_2(1-q)}\rb^N\zt.$$ Write:
\begin{equation}\label{eq:partition function}
    \begin{split}
        \zt&=\mathbb{E}(1+rt+2\sqrt{rt}Y_{rt})^N\\
        &=\int_{-\infty}^\infty(1+rt+2\sqrt{rt}y)^N\nu(dy,\ttA,\ttB,\ttC,\ttD,q)\\
        &=\int_{-1}^1(1+rt+2\sqrt{rt}y)^Nf(y,\ttA,\ttB,\ttC,\ttD,q)dy\\
        &\quad+\sum_{y_j(t)\in F(t)}(1+rt+2\sqrt{rt}y_j(t))^Np(y_j(t);\ttA,\ttB,\ttC,\ttD,q),
    \end{split}
\end{equation}
where $$y_j(t)\in F(t)=F(\ttA,\ttB,\ttC,\ttD,q)$$ are atoms generated by $\ttA,\ttB,\ttC,\ttD$ and
\begin{equation*}
\begin{split}
    &f(y,\ttA,\ttB,\frac{\tC}{\sqrt{t}},\frac{\tD}{\sqrt{t}},q)\\
    &=\frac{(q,t\tA\tB,\tA\tC,\tA\tD,\tB\tC,\tB\tD,\tC\tD/t;q)_{\infty}}{2\pi(\tA\tB\tC\tD;q)_{\infty}\sqrt{1-y^2}}\left\vert\frac{(e^{2i\th_y};q)_{\infty}}{(\ttA e^{i\th_y},\ttB e^{i\th_y},\frac{\tC}{\sqrt{t}} e^{i\th_y},\frac{\tD}{\sqrt{t}} e^{i\th_y};q)_{\infty}}\right\vert^2
\end{split} 
\end{equation*}
is the continuous part density, where $y=\cos\th_y\in[-1,1]$. In the following we denote
\be   \label{eq:definition of pi t y}
\pi_t(y):=f(y,\ttA,\ttB,\ttC,\ttD,q).
\ee

We make some observations on Askey-Wilson measure $\nu(dy,\ttA,\ttB,\ttC,\ttD,q)$ as $t\rightarrow 1$. We first look at the behavior of atoms. By \eqref{eq:conditions on tAtBtCtD} all the possible atoms are generated by $\tA\sqrt{t},\ttC$ and $\ttD$. In the high density phase $\tA>1$ there are atoms generated by $\tA\sqrt{t}$  and in the low density phase $\tC>1$ there are atoms generated by $\ttC$. In all three phases there may be atoms generated by $\ttD$.  
By the technical condition \eqref{eq:technical condition}, for $t$ in a small neighborhood of $1$, the number of atoms is constant. The positions of atoms $y_j(t)$ and the corresponding masses $p(y_j(t);\ttA,\ttB,\ttC,\ttD,q)$ are  smooth functions of $t$. 
\begin{fact}\label{fact in proof discrete}
    The Askey-Wilson measure $\nu(dy,\tA,\tB,\tC,\tD,q)$ is supported inside $(-(1+r)/2\sqrt{r},\infty)$, where the function $(1+r+2\sqrt{r}y)^N$ of $y$ is strictly positive and strictly increasing.
\end{fact}
\begin{proof}[Proof of Fact \ref{fact in proof discrete}]
    The continuous part is supported on $[-1,1]$. The atoms generated by $\tA$ and by $\tC$ lie in $[1,\infty)$. When $\tD$ generates atoms, we have $\tD\in(-1/\sqrt{r},-1)$, hence the atoms generated by $\tD$ are bounded below by 
$\frac{1}{2}(q^i\tD+\frac{1}{q^i\tD})\geq\frac{1}{2}(\tD+\frac{1}{\tD})>-\frac{1+r}{2\sqrt{r}}$, for $i=0,1,\dots$
\end{proof}

We then look at the behavior of continuous part density.
\begin{fact}\label{fact in the proof continuous}
     There is a smooth function $g(t,z)$  on a small neighborhood of $\left\{1\right\}\times\left\{|z|=1\right\}\subset\mathbb{C}\times\mathbb{C}$ which cannot take the value $0$, such that:
     \be  \label{eq: continuous part density other expression}
    \pi_t(y)=\sqrt{1-y^2}g(t,z),
    \ee
    where $z=e^{i\th_y}$ for $y\in[-1,1]$. In particular both $\pi_t(y)$ and $\pt\pi_t(y)$
    are both bounded on $(t,y)\in(1-\ep,1+\ep)\times[-1,1]$ for some $\ep>0$, hence 
    $\pt\lb(1+rt+2\sqrt{rt}y)^N\pi_t(y)\rb$ is also bounded on this region. 
    Hence  we can take differentiation under the integral sign in $\pt\lb\int_{-1}^1(1+rt+2\sqrt{rt}y)^N\pi_t(y)dy\rb\po$.
\end{fact}
\begin{proof} [Proof of Fact \ref{fact in the proof continuous}]
    In numerator of $\pi_t(y)$ we have
$|(e^{2i\th_y};q)_\infty|^2=4(1-y^2)|(e^{2i\th_y}q;q)_\infty|^2$. Define:
$$g(t,z)=\frac{2(q,t\tA\tB,\tA\tC,\tA\tD,\tB\tC,\tB\tD,\tC\tD/t;q)_{\infty}}{\pi(\tA\tB\tC\tD;q)_{\infty}}\left\vert\frac{(qz^2;q)_{\infty}}{(\ttA z,\ttB z,\frac{\tC}{\sqrt{t}}z,\frac{\tD}{\sqrt{t}}z;q)_{\infty}}\right\vert^2.$$
By condition \eqref{eq:technical condition}, when $t$ and $|z|$ are close to $1$, $|1-sq^j|>\ep$ for some $\ep>0$, $s=\ttA z,\ttB z,\frac{\tC}{\sqrt{t}}z,\frac{\tD}{\sqrt{t}}z$ and $j=0,1,2,\dots$ Hence by the analyticity of $q$-Pochhammer symbols, $g(t,z)$ is smooth.
\end{proof}

We return to the proof of the theorem. 

\begin{enumerate}
    \item [(\upperRomannumeral{1})] (High density phase $\tA>1$).
    When $t\rightarrow 1$ there are atoms generated by $\ttA$ and also possible atoms generated by $\ttD$. 
    By Facts \ref{fact in proof discrete} and \ref{fact in the proof continuous} we can observe that as $N\rightarrow\infty$ both $\zn$ and $\pt\zt\po$ are dominated by largest atom 
    $y_0(t)=\frac{1}{2}\lb\tA\sqrt{t}+\frac{1}{\tA\sqrt{t}}\rb,$ i.e.
\be  \label{eq:partition function high density}
\zn\sim (1+r+2\sqrt{r}y_0(1))^Np(y_0(1),\tA,\tB,\tC,\tD,q),
\ee
\be  \label{eq:partition function derivative high density}
\pt\zt\po\sim N\pt(1+rt+2\sqrt{rt}y_0(t))\po(1+r+2\sqrt{r}y_0(1))^{N-1}p(y_0(1),\tA,\tB,\tC,\tD,q),
\ee
where we use $f(N)\sim g(N)$ to denote $\lim_{N\rightarrow\infty}f(N)/g(N)=1$. The details are provided in Appendix \ref{section:appendix one}. 
Hence we have
$$\lim_{N\rightarrow\infty}\mathbb{E}_\mu\rho=\frac{\pt(1+rt+2\sqrt{rt}y_0(t))\po}{(1+r+2\sqrt{r}y_0(1))}=\frac{\pt(1+rt+Art+1/A)\po}{1+r+Ar+1/A}=\frac{Ar}{Ar+1}.$$
\item [(\upperRomannumeral{2})] (Low density phase $\tC>1$). The proof is similar to the high density phase. Both $\zn$ and $\pt\zt\po$ are dominated by the largest atom $y_0(t)=\frac{1}{2}\lb\frac{\tC}{\sqrt{t}}+\frac{\sqrt{t}}{\tC}\rb.$ Hence
$$\lim_{N\rightarrow\infty}\mathbb{E}_\mu\rho=\frac{\pt(1+rt+2\sqrt{rt}y_0(t))\po}{(1+r+2\sqrt{r}y_0(1))}=\frac{\pt(1+rt+C+rt/C)\po}{1+r+C+r/C}=\frac{r}{C+r}.$$
\item [(\upperRomannumeral{3})] (Maximal current phase $\tA,\tC<1$). The only possible atoms come from $\ttD$.
By Fact \ref{fact in the proof continuous} we have $|\pt \pi_t(y)\po|\leq M\sqrt{1-y^2}$ and $\pi_1(y)>\ep\sqrt{1-y^2}$
on $y\in[-1,1]$, for some $M,\ep>0$.
We can observe that as $N\rightarrow\infty$ we have:
\be   \label{eq:partition function maximal current}
\zn\sim\int_{-1}^1(1+r+2\sqrt{r}y)^N\pi_1(y)dy,
\ee
\be   \label{eq:partition function derivative maximal current}
\frac{\pt\zt\po}{N\zn}\sim\frac{\int_{-1}^1(1+r+2\sqrt{r}y)^{N-1}(r+\sr y)\pi_1(y)dy}{\int_{-1}^1(1+r+2\sqrt{r}y)^N\pi_1(y)dy}.
\ee   
The details are provided in Appendix \ref{section:appendix two}.

We use a similar method as \cite[section 4.2]{BW19} to obtain the asymptotics.
We first write
\begin{equation*}
\begin{split}
    \pi_1(y)&=\frac{(q,\tA\tB,\tA\tC,\tA\tD,\tB\tC,\tB\tD,\tC\tD;q)_{\infty}}{2\pi(\tA\tB\tC\tD;q)_{\infty}\sqrt{1-y^2}}\left\vert\frac{(e^{2i\th_y};q)_{\infty}}{(\tA e^{i\th_y},\tB e^{i\th_y},\tC e^{i\th_y},\tD e^{i\th_y};q)_{\infty}}\right\vert^2\\
    &=\sqrt{1-y^2}\frac{2(q,\tA\tB,\tA\tC,\tA\tD,\tB\tC,\tB\tD,\tC\tD;q)_{\infty}}{\pi(\tA\tB\tC\tD;q)_{\infty}\vert(\tA e^{i\th_y},\tB e^{i\th_y},\tC e^{i\th_y},\tD e^{i\th_y};q)_{\infty}\vert^2}|(qe^{2i\th_y};q)_\infty|^2.
\end{split}
\end{equation*}
Set $y=1-\frac{u}{2N}$. Fix $u>0$, as $N\rightarrow\infty$ we have $e^{i\th_y}\rightarrow 1$ and $(qe^{2i\th_y};q)_\infty\rightarrow (q;q)_\infty$. Hence
$$\pi_1(y)\sim 2\mathfrak{c}\sqrt{\frac{u}{N}}, \quad\text{and}\quad\pi_1(y)\leq M\sqrt{\frac{u}{N}},$$
where 
$$\mathfrak{c}=\frac{(q;q)_\infty^3(\tA\tB,\tA\tC,\tA\tD,\tB\tC,\tB\tD,\tC\tD;q)_\infty}{\pi(\tA\tB\tC\tD;q)_{\infty}(\tA,\tB,\tC,\tD;q)_\infty^2},$$
and $M$ is a large constant (may be different from $M$ used before). We write:
\begin{equation*}
    \begin{split}
        &\int_{-1}^1(1+r+2\sqrt{r}y)^N\pi_1(y)dy\\&=\int_0^\infty 1_{u\leq 4N}\lb 1+r+2\sr-\frac{u\sr}{N}\rb^N\pi_1\lb 1-\frac{u}{2N}\rb\frac{du}{2N}\\
        &=\frac{(1+\sr)^{2N}}{2N^{\frac{3}{2}}}\int_0^\infty 1_{u\leq 4N}\lb 1-\frac{u\sr}{N(1+\sr)^2}\rb^N\pi_1\lb 1-\frac{u}{2N}\rb \sqrt{N} du.
    \end{split}
\end{equation*}
As $N\rightarrow\infty$ the integrand is bounded above by a constant times $\exp\lb-\frac{u\sr}{(1+\sr)^2}\rb\sqrt{u}$, which is integrable on $u\in(0,\infty)$. We can use dominated convergence theorem to take $N\rightarrow\infty$:
\begin{equation*}
    \begin{split}
        \int_{-1}^1(1+r+2\sqrt{r}y)^N\pi_1(y)dy&\sim\frac{(1+\sr)^{2N}}{N^{\frac{3}{2}}}\mathfrak{c}\int_0^\infty \exp\lb-\frac{u\sr}{(1+\sr)^2}\rb\sqrt{u} du\\
        &=2\mathfrak{c}\frac{(1+\sr)^{2N}}{N^{\frac{3}{2}}}\int_0^\infty \exp\lb-\frac{\sr}{(1+\sr)^2}t^2\rb t^2dt\\
        &=\frac{\sqrt{\pi}}{2}\mathfrak{c}\frac{(1+\sr)^{2N+3}}{N^{\frac{3}{2}}r^{\frac{3}{4}}},
    \end{split}
\end{equation*}
where we have used $\int_0^\infty e^{-st^2}t^2dt=\frac{1}{4}\sqrt{\frac{\pi}{s^3}}$ for $s>0$ and took $s=\frac{\sr}{(1+\sr)^2}$. Hence:
\begin{equation*}
  \begin{split}
      \mathbb{E}_\mu\rho=\frac{\pt\zt\po}{N\zn} &\sim\frac{\int_{-1}^1(1+r+2\sqrt{r}y)^{N-1}(r+\sr)\pi_1(y)dy}{\int_{-1}^1(1+r+2\sqrt{r}y)^N\pi_1(y)dy}\\
      &=\frac{1}{2}+\frac{r-1}{2}\frac{\int_{-1}^1(1+r+2\sqrt{r}y)^{N-1}\pi_1(y)dy}{\int_{-1}^1(1+r+2\sqrt{r}y)^N\pi_1(y)dy}\\
      &\sim\frac{1}{2}+\frac{r-1}{2(1+\sqrt{r})^2}=\frac{\sr}{1+\sr}.
  \end{split}
\end{equation*}

\end{enumerate}

We show that under an extra condition $a+c<1$ we can avoid technical condition \eqref{eq:technical condition}:
$$|A\sqrt{r}|,|B\sqrt{r}|,|C/\sqrt{r}|,|D/\sqrt{r}|\neq q^{-i} \text{ for }i=0,1,2,\dots$$ 
Denote the right hand side of \eqref{eq:limit of mean density S6V} as $\rho(A,C)$, which is a continuous function of $A,C$ and hence a continuous function of $a,b,c,d,\th_1,\th_2$. 
Suppose $(a,b,c,d,\th_1,\th_2)$ satisfy \eqref{eq:condition Thm 1.1}, $AC<1$ and $a+c<1$.
We can choose a sequence $(a_n,b_n,c_n,d_n)\in(0,1)^4$, $n=1,2\dots$ approaching $(a,b,c,d)$, such that 
$$a_n\leq a,\quad b_n\geq b,\quad c_n\geq c,\quad d_n\leq d,\quad a+c_n<1,\quad b_n+d<1,$$
and the corresponding parameters $(A_n,B_n,C_n,D_n)$ (Definition \ref{defn:parameterization open ASEP}) satisfy $A_nC_n<1$ and
$$|A_n\sqrt{r}|,|B_n\sqrt{r}|,|C_n/\sqrt{r}|,|D_n/\sqrt{r}|\neq q^{-i} \text{ for }i=0,1,2,\dots$$ 
Denote by $\Pi_N^n$ the six-vertex model on a strip with width $N$ and with parameters $(a_n,b_n,c_n,d_n,\th_1,\th_2)$, for $n=1,2\dots$ Denote the six-vertex model on a strip with width $N$ with parameters $(a,b,c,d,\th_1,\th_2)$ by $\Pi_N$. Assume these models have empty initial condition. By Proposition \ref{prop:coupling} we have a coupling of $\Pi_N^n$ with $\Pi_N$ such that their occupation variables satisfy $\eta_N^n(\se)\leq\eta_N(\se)$ for every edge $\se$ of the strip, where $\eta_N^n$ is the occupation variable of $\Pi_N^n$ and $\eta_N$ is the occupation variable of $\Pi_N$. We then consider the interacting particle systems on a horizontal path related to those six-vertex models. Since they are finite ergodic Markov chains, as time (vertical coordinate) goes to infinity they converge to their stationary measures $\mu_N^n$ and $\mu_N$. 
In particular 
$\mathbb{E}_{\mu_N^n}\rho\leq\mathbb{E}_{\mu_N}\rho$ for every $n,N\in\mathbb{Z}_+$. 
Since the parameters $(a_n,b_n,c_n,d_n,\th_1,\th_2)$ satisfy technical condition \eqref{eq:technical condition},
we have shown that $\lim_{N\rightarrow\infty}\mathbb{E}_{\mu_N^n}\rho=\rho(A_n,C_n)$.
Hence $\varliminf_{N\rightarrow\infty}\mathbb{E}_{\mu_N}\rho\geq\lim_{N\rightarrow\infty}\mathbb{E}_{\mu_N^n}\rho=\rho(A_n,C_n)$.
We then take $n\rightarrow\infty$ and get $\varliminf_{N\rightarrow\infty}\mathbb{E}_{\mu_N}\rho\geq\lim_{N\rightarrow\infty}\rho(A_n,C_n)=\rho(A,C)$. 
By the same argument we also have 
$\varlimsup_{N\rightarrow\infty}\mathbb{E}_{\mu_N}\rho\leq\rho(A,C)$.
Hence $\lim_{N\rightarrow\infty}\mathbb{E}_{\mu_N}\rho=\rho(A,C)$.
\end{proof}

\begin{proposition}[Two-species attractive coupling] \label{prop:coupling}
    Suppose  $a,b,c,d,a',b',c',d',\th_1,\th_2\in(0,1)$ such that
    $$a\leq a',\quad b\geq b',\quad c\geq c',\quad d\leq d',\quad a'+c<1,\quad b+d'<1.$$
    Consider two six-vertex models $\Pi_1,\Pi_2$ on a strip with the same width $N$, with parameters respectively $(a,b,c,d,\th_1,\th_2)$ and $(a',b',c',d',\th_1,\th_2)$. Denote the occupation variables of $\Pi_1$ and $\Pi_2$ by $\eta_{1}(\se)\in\left\{0,1\right\}$ and $\eta_{2}(\se)\in\left\{0,1\right\}$, where $\se$ runs through edges of the strip. 
    
    Suppose that $\eta_{1}(\se)\leq\eta_{2}(\se)$ for any edge $\se$ in the initial condition (outgoing edge of a down-right path $\hP$). Then there exists a coupling of $\Pi_1$ and $\Pi_2$ such that $\eta_{1}(\se)\leq\eta_{2}(\se)$ for any edge $\se$.
\end{proposition}
\begin{proof}
    We construct a six-vertex model on the strip $\Pi$ with arrows in two colors $1,2$, and we use $0$ to denote the absence of an arrow. We then show that two marginals of this model equal respectively $\Pi_1$ and $\Pi_2$. Denote the occupation variable of $\Pi$ by $\eta(\se)\in\left\{0,1,2\right\}$, where $\se$ runs through edges of the strip. The initial condition of $\Pi$ is defined by 
    $\eta(\se)=\eta_1(\se)+\eta_2(\se)$
    for all initial edges $\se$.
    
    The sampling dynamics of the model $\Pi$ is defined by the following vertex weights:
    \begin{enumerate}
        \item [$\bullet$] Bulk weights: For every $0\leq i\leq 2$ we have
        $$\PP\lb\raisebox{-14pt}{\begin{tikzpicture}[scale=0.3]
		\draw[thick] (-1,0) -- (1,0);
		\draw[thick] (0,-1) -- (0,1);
            \node[above] at (0,1) {\tiny $i$};
            \node[below] at (0,-1) {\tiny $i$};
            \node[left] at (-1,0) {\tiny $i$};
            \node[right] at (1,0) {\tiny $i$};
		\end{tikzpicture}}\rb=1.$$
  For every pair of $0\leq i<j\leq 2$ we have
  $$\PP\lb\raisebox{-14pt}{\begin{tikzpicture}[scale=0.3]
		\draw[thick] (-1,0) -- (1,0);
		\draw[thick] (0,-1) -- (0,1);
            \node[above] at (0,1) {\tiny $i$};
            \node[below] at (0,-1) {\tiny $i$};
            \node[left] at (-1,0) {\tiny $j$};
            \node[right] at (1,0) {\tiny $j$};
		\end{tikzpicture}}\rb=\th_2,   \quad
  \PP\lb\raisebox{-14pt}{\begin{tikzpicture}[scale=0.3]
		\draw[thick] (-1,0) -- (1,0);
		\draw[thick] (0,-1) -- (0,1);
            \node[above] at (0,1) {\tiny $j$};
            \node[below] at (0,-1) {\tiny $i$};
            \node[left] at (-1,0) {\tiny $j$};
            \node[right] at (1,0) {\tiny $i$};
		\end{tikzpicture}}\rb=1-\th_2,  $$$$
  \PP\lb\raisebox{-14pt}{\begin{tikzpicture}[scale=0.3]
		\draw[thick] (-1,0) -- (1,0);
		\draw[thick] (0,-1) -- (0,1);
            \node[above] at (0,1) {\tiny $j$};
            \node[below] at (0,-1) {\tiny $j$};
            \node[left] at (-1,0) {\tiny $i$};
            \node[right] at (1,0) {\tiny $i$};
		\end{tikzpicture}}\rb=\th_1,  \quad
  \PP\lb\raisebox{-14pt}{\begin{tikzpicture}[scale=0.3]
		\draw[thick] (-1,0) -- (1,0);
		\draw[thick] (0,-1) -- (0,1);
            \node[above] at (0,1) {\tiny $i$};
            \node[below] at (0,-1) {\tiny $j$};
            \node[left] at (-1,0) {\tiny $i$};
            \node[right] at (1,0) {\tiny $j$};
		\end{tikzpicture}}\rb=1-\th_1.    $$
  \item[$\bullet$] Right boundary weights:
  $$\PP\lb\raisebox{-14pt}{\begin{tikzpicture}[scale=0.5]
		\draw[thick] (0,0) -- (0,1);
		\draw[thick] (-1,0) -- (0,0);
            \node[left] at (-1,0) {\tiny$0$};
            \node[above] at (0,1) {\tiny$0$};
		\end{tikzpicture}}\rb=1-d',\quad
  \PP\lb\raisebox{-14pt}{\begin{tikzpicture}[scale=0.5]
		\draw[thick] (0,0) -- (0,1);
		\draw[thick] (-1,0) -- (0,0);
            \node[left] at (-1,0) {\tiny$0$};
            \node[above] at (0,1) {\tiny$1$};
		\end{tikzpicture}}\rb=d'-d,\quad
  \PP\lb\raisebox{-14pt}{\begin{tikzpicture}[scale=0.5]
		\draw[thick] (0,0) -- (0,1);
		\draw[thick] (-1,0) -- (0,0);
            \node[left] at (-1,0) {\tiny$0$};
            \node[above] at (0,1) {\tiny$2$};
		\end{tikzpicture}}\rb=d,$$

    $$\PP\lb\raisebox{-14pt}{\begin{tikzpicture}[scale=0.5]
		\draw[thick] (0,0) -- (0,1);
		\draw[thick] (-1,0) -- (0,0);
            \node[left] at (-1,0) {\tiny$1$};
            \node[above] at (0,1) {\tiny$0$};
		\end{tikzpicture}}\rb=b',\quad
  \PP\lb\raisebox{-14pt}{\begin{tikzpicture}[scale=0.5]
		\draw[thick] (0,0) -- (0,1);
		\draw[thick] (-1,0) -- (0,0);
            \node[left] at (-1,0) {\tiny$1$};
            \node[above] at (0,1) {\tiny$1$};
		\end{tikzpicture}}\rb=1-b'-d,\quad
  \PP\lb\raisebox{-14pt}{\begin{tikzpicture}[scale=0.5]
		\draw[thick] (0,0) -- (0,1);
		\draw[thick] (-1,0) -- (0,0);
            \node[left] at (-1,0) {\tiny$1$};
            \node[above] at (0,1) {\tiny$2$};
		\end{tikzpicture}}\rb=d,$$

    $$\PP\lb\raisebox{-14pt}{\begin{tikzpicture}[scale=0.5]
		\draw[thick] (0,0) -- (0,1);
		\draw[thick] (-1,0) -- (0,0);
            \node[left] at (-1,0) {\tiny$2$};
            \node[above] at (0,1) {\tiny$0$};
		\end{tikzpicture}}\rb=b',\quad
  \PP\lb\raisebox{-14pt}{\begin{tikzpicture}[scale=0.5]
		\draw[thick] (0,0) -- (0,1);
		\draw[thick] (-1,0) -- (0,0);
            \node[left] at (-1,0) {\tiny$2$};
            \node[above] at (0,1) {\tiny$1$};
		\end{tikzpicture}}\rb=b-b',\quad
  \PP\lb\raisebox{-14pt}{\begin{tikzpicture}[scale=0.5]
		\draw[thick] (0,0) -- (0,1);
		\draw[thick] (-1,0) -- (0,0);
            \node[left] at (-1,0) {\tiny$2$};
            \node[above] at (0,1) {\tiny$2$};
		\end{tikzpicture}}\rb=1-b.$$
  \item[$\bullet$] Left boundary weights:
    $$\PP\lb\raisebox{-14pt}{\begin{tikzpicture}[scale=0.5]
		\draw[thick] (0,-1) -- (0,0);
		\draw[thick] (0,0) -- (1,0);
            \node[below] at (0,-1) {\tiny$0$};
            \node[right] at (1,0) {\tiny$0$};
		\end{tikzpicture}}\rb=1-a',\quad
  \PP\lb\raisebox{-14pt}{\begin{tikzpicture}[scale=0.5]
		\draw[thick] (0,-1) -- (0,0);
		\draw[thick] (0,0) -- (1,0);
            \node[below] at (0,-1) {\tiny$0$};
            \node[right] at (1,0) {\tiny$1$};
		\end{tikzpicture}}\rb=a'-a,\quad
  \PP\lb\raisebox{-14pt}{\begin{tikzpicture}[scale=0.5]
		\draw[thick] (0,-1) -- (0,0);
		\draw[thick] (0,0) -- (1,0);
            \node[below] at (0,-1) {\tiny$0$};
            \node[right] at (1,0) {\tiny$2$};
		\end{tikzpicture}}\rb=a,$$
  $$\PP\lb\raisebox{-14pt}{\begin{tikzpicture}[scale=0.5]
		\draw[thick] (0,-1) -- (0,0);
		\draw[thick] (0,0) -- (1,0);
            \node[below] at (0,-1) {\tiny$1$};
            \node[right] at (1,0) {\tiny$0$};
		\end{tikzpicture}}\rb=c',\quad
  \PP\lb\raisebox{-14pt}{\begin{tikzpicture}[scale=0.5]
		\draw[thick] (0,-1) -- (0,0);
		\draw[thick] (0,0) -- (1,0);
            \node[below] at (0,-1) {\tiny$1$};
            \node[right] at (1,0) {\tiny$1$};
		\end{tikzpicture}}\rb=1-c'-a,\quad
  \PP\lb\raisebox{-14pt}{\begin{tikzpicture}[scale=0.5]
		\draw[thick] (0,-1) -- (0,0);
		\draw[thick] (0,0) -- (1,0);
            \node[below] at (0,-1) {\tiny$1$};
            \node[right] at (1,0) {\tiny$2$};
		\end{tikzpicture}}\rb=a,$$
  $$\PP\lb\raisebox{-14pt}{\begin{tikzpicture}[scale=0.5]
		\draw[thick] (0,-1) -- (0,0);
		\draw[thick] (0,0) -- (1,0);
            \node[below] at (0,-1) {\tiny$2$};
            \node[right] at (1,0) {\tiny$0$};
		\end{tikzpicture}}\rb=c',\quad
  \PP\lb\raisebox{-14pt}{\begin{tikzpicture}[scale=0.5]
		\draw[thick] (0,-1) -- (0,0);
		\draw[thick] (0,0) -- (1,0);
            \node[below] at (0,-1) {\tiny$2$};
            \node[right] at (1,0) {\tiny$1$};
		\end{tikzpicture}}\rb=c-c',\quad
  \PP\lb\raisebox{-14pt}{\begin{tikzpicture}[scale=0.5]
		\draw[thick] (0,-1) -- (0,0);
		\draw[thick] (0,0) -- (1,0);
            \node[below] at (0,-1) {\tiny$2$};
            \node[right] at (1,0) {\tiny$2$};
		\end{tikzpicture}}\rb=1-c.$$
    \end{enumerate}
    
We can check that two marginals 
$$\eta_1(\se):=1_{\eta(\se)\geq 2},\quad\eta_2(\se):=1_{\eta(\se)\geq 1},\quad\forall\text{ edge }\se \text{ on the strip }$$ 
of the above model $\Pi$ coincide respectively with $\Pi_1$ and $\Pi_2$. It is immediate that $\eta_1(\se)\leq\eta_2(\se)$.
\end{proof}
\appendix\section{Details in the proof of Theorem \ref{thm:main theorem section 3}}
\subsection{Proofs of \texorpdfstring{\eqref{eq:partition function high density} and \eqref{eq:partition function derivative high density}}{} in the high density phase}\label{section:appendix one}
 On high density phase $\tA>1$, when $t$ is in a small interval of $1$, the measure $\nu(dy,\ttA,\ttB,\ttC,\ttD,q)$ has atoms 
$y_0(t)>\dots>y_{\ell-1}(t)>1$ generated by $\ttA$ and also possibly some atoms $-1>y_{\ell}(t)>\dots>y_{k}(t)$ generated by $\ttD$, where $k\geq \ell-1\geq 0$.
 We take $t=1$ in $\zt$ given by \eqref{eq:partition function} and get:

\begin{equation*} 
\begin{split}
        \zn=&(1+r+2\sqrt{r}y_0(1))^Np(y_0(1);\tA,\tB,\tC,\tD,q)\\
        &+\sum_{j=1}^{k}(1+r+2\sqrt{r}y_j(1))^Np(y_j(1);\tA,\tB,\tC,\tD,q)+\int_{-1}^1(1+r+2\sqrt{r}y)^N\pi_1(y)dy,
\end{split}
\end{equation*}
where $\pi_t(y)$ is defined in \eqref{eq:definition of pi t y}.
The first term divided by $(1+r+2\sqrt{r}y_0(1))^N$ is a positive constant independent of $N$. By Fact \ref{fact in proof discrete}, every other terms divided by $(1+r+2\sqrt{r}y_0(1))^N$ goes to $0$ as $N\rightarrow\infty$.
 Hence we obtain \eqref{eq:partition function high density}:
$$\zn\sim (1+r+2\sqrt{r}y_0(1))^Np(y_0(1),\tA,\tB,\tC,\tD,q).$$

Next we take the derivative of $\zt$ at $t=1$.  
We have:
\begin{equation*} 
\begin{split}
        &\pt\zt\po\\=& N\pt(1+rt+2\sqrt{rt}y_0(t))\po(1+r+2\sqrt{r}y_0(1))^{N-1} p(y_0(1);\tA,\tB,\tC,\tD,q)\\
        &+(1+r+2\sqrt{r}y_0(1))^N \pt p(y_0(t);\ttA,\ttB,\ttC,\ttD,q)\po\\
        &+\sum_{j=1}^{k}N\pt(1+rt+2\sqrt{rt}y_j(t))\po(1+r+2\sqrt{r}y_j(1))^{N-1} p(y_j(1);\tA,\tB,\tC,\tD,q)\\
        &+\sum_{j=1}^{k}(1+r+2\sqrt{r}y_j(1))^N \pt p(y_j(t);\ttA,\ttB,\ttC,\ttD,q)\po\\
        &+\int_{-1}^1 N\pt(1+rt+2\sqrt{rt}y)\po (1+r+2\sqrt{r}y)^{N-1}\pi_1(y)dy\\
        &+\int_{-1}^1(1+r+2\sqrt{r}y)^N\pt\pi_t(y)\po dy.
\end{split}
\end{equation*}
The first term divided by $N(1+r+2\sqrt{r}y_0(1))^{N}$ is a positive constant independent of $N$, and every other terms divided by $N(1+r+2\sqrt{r}y_0(1))^{N}$ goes to $0$ as $N\rightarrow\infty$.
Hence we obtain \eqref{eq:partition function derivative high density}:
$$\pt\zt\po\sim N\pt(1+rt+2\sqrt{rt}y_0(t))\po(1+r+2\sqrt{r}y_0(1))^{N-1}p(y_0(1),\tA,\tB,\tC,\tD,q).$$
 \subsection{Proof of \texorpdfstring{\eqref{eq:partition function maximal current} and \eqref{eq:partition function derivative maximal current}}{} in the maximal current phase}\label{section:appendix two}
On the maximal current phase $\tA<1$, $\tC<1$, when $t$ is close to $1$ the measure $\nu(dy,\ttA,\ttB,\ttC,\ttD,q)$ can only have possible atoms $-1>y_0(t)>\dots>y_{k}(t)$ generated by $\ttD$,
where $k\geq -1$. 
 We have:
$$\zn=\int_{-1}^1(1+r+2\sqrt{r}y)^N\pi_1(y)dy+\sum_{j=0}^k(1+r+2\sqrt{r}y_j(1))^Np(y_j(1);\tA,\tB,\tC,\tD,q)=C_1+C_2.$$
Observe that the density $\pi_1(y)$ can only take zero at $y=-1,1$.
When $N\rightarrow\infty$, the integral $C_1$ divided by $(1+r-2\sqrt{r})^N$ is bounded below by $\int_{-1}^1\pi_1(y)dy>0$,
and each summand in  $C_2$ divided by $(1+r-2\sqrt{r})^N$
 goes to $0$. We obtain \eqref{eq:partition function maximal current}:
$$\zn\sim\int_{-1}^1(1+r+2\sqrt{r}y)^N\pi_1(y)dy.$$
Next we evaluate $\pt\zt\po$ and look at $\frac{\pt\zt\po}{N\zn}$ as $N\rightarrow\infty$:
\begin{equation*}
    \begin{split}
        &\pt\zt\po\\=&\int_{-1}^1 N(r+\sqrt{r}y) (1+r+2\sqrt{r}y)^{N-1}\pi_1(y)dy
        +\int_{-1}^1(1+r+2\sqrt{r}y)^N\pt\pi_t(y)\po dy\\
        &+\sum_{j=0}^{k}N\pt(1+rt+2\sqrt{rt}y_j(t))\po(1+r+2\sqrt{r}y_j(1))^{N-1} p(y_j(1);\tA,\tB,\tC,\tD,q)\\
        &+\sum_{j=0}^{k}(1+r+2\sqrt{r}y_j(1))^N \pt p(y_j(t);\ttA,\ttB,\ttC,\ttD,q)\po\\
        =&D_1+D_2+D_3+D_4.
    \end{split}
\end{equation*}
Note that $\zn\geq\int_{-1}^1(1+r+2\sqrt{r}y)^N\pi_1(y)dy.$
When $N\rightarrow\infty$, $N\zn$ divided by $N(1+r-2\sqrt{r})^N$ is bounded below by $\int_{-1}^1\pi_1(y)dy>0$, and each summand in $D_3$ and $D_4$ divided by $N(1+r-2\sqrt{r})^N$ goes to $0$.
For the integral $D_2$, observe that 
by Fact \ref{fact in the proof continuous} we have $|\pt \pi_t(y)\po|\leq M\sqrt{1-y^2}$ and $\pi_1(y)>\ep\sqrt{1-y^2}$ for some $M,\ep>0$. Hence on $y\in[-1,1]$ we have
$$\left\vert(1+r+2\sqrt{r}y)^N\pt\pi_t(y)\po\right\vert\leq \frac{M}{\ep}(1+r+2\sqrt{r}y)^N\pi_1(y).$$
 Integrate over $[-1,1]$, we get
$$\int_{-1}^1|(1+r+2\sqrt{r}y)^N\pt\pi_t(y)\po| dy\leq\frac{M}{\ep}\int_{-1}^1(1+r+2\sqrt{r}y)^N\pi_1(y)dy.$$
Hence
$$\left\vert\frac{D_2}{N\zn}\right\vert\leq\left\vert\frac{\int_{-1}^1(1+r+2\sqrt{r}y)^N\pt\pi_t(y)\po dy}{N\int_{-1}^1(1+r+2\sqrt{r}y)^N\pi_1(y)dy}\right\vert\leq \frac{M/\ep}{N}\rightarrow 0.$$
Hence we obtain \eqref{eq:partition function derivative maximal current}:
$$\frac{\pt\zt\po}{N\zn}\sim\frac{D_1}{N\zn}=\frac{\int_{-1}^1(1+r+2\sqrt{r}y)^{N-1}(r+\sr y)\pi_1(y)dy}{\int_{-1}^1(1+r+2\sqrt{r}y)^N\pi_1(y)dy}.$$

\end{document}